\documentclass[a4paper,leqno]{amsart}

\usepackage[english]{babel}
\usepackage[utf8]{inputenc}
\usepackage[T1]{fontenc}
\usepackage{amsmath}
\usepackage{amssymb}
\usepackage{multirow}
\usepackage{alltt}
\usepackage{graphicx}
\usepackage{bbm}
\usepackage{amsthm}
\usepackage{mathrsfs}
\usepackage{textcomp}
\usepackage{calc}
\usepackage{fullpage}
\usepackage{amsrefs}
\usepackage{enumitem}
\usepackage{epstopdf}

\usepackage[usenames,dvipsnames]{xcolor}
\usepackage[hyperindex, linktocpage]{hyperref}
\hypersetup{colorlinks=true, linkcolor=Violet, citecolor=BrickRed}
\usepackage{tikz}

\newcommand{\R}{\mathbb{R}}
\newcommand{\N}{\mathbb{N}}

\newcommand{\eps}{\varepsilon}

\newcommand{\pv}{\mathrm{p.v.}\!\!}
\newcommand{\Ds}{{\left(-\lapl\right)}^s}
\newcommand{\lapl}{\Delta}

\newcommand{\1}{\mathbf{1}}
\newcommand{\hf}{\,{}_2F_1}

\newcommand{\cH}{{\mathcal H}}

\renewcommand{\epsilon}{\varepsilon}

\newtheorem{theorem}{Theorem}[section]
\newtheorem{lemma}[theorem]{Lemma}
\newtheorem{proposition}[theorem]{Proposition}

\theoremstyle{remark}
\newtheorem{remark}[theorem]{Remark}

\theoremstyle{definition}

\allowdisplaybreaks
\numberwithin{equation}{section}

\title{On the shape of the first fractional eigenfunction}

\date{\today}

\author{Nicola Abatangelo}
\address{(N. Abatangelo) Dipartimento di Matematica, Alma Mater Studiorum Università di Bologna, P.zza di Porta S. Donato 5, 40126 Bologna, Italy.}
\email{nicola.abatangelo@unibo.it}
\author{Sven Jarohs}
\address{(S. Jarohs) Institut für Mathematik, Goethe-Universität Frankfurt am Main, Robert-Meyer-Str.~10, 60325 Frankfurt am Main, Germany.}
\email{jarohs@math.uni-frankfurt.de}

\thanks{ 
{\it Secondary:} 
47G20, 
35R11. 
}

\thanks{{\it Acknowledgements}: The first author has been partially supported by the Alexander von Humboldt Foundation.}

\begin{document}

\begin{abstract}
We show that the first eigenfunction of the fractional Laplacian~$\Ds$,~$s\in(1/2,1)$, 
is superharmonic in the unitary ball up to dimension $11$. 
To this aim, we also rely on a computer-assisted step to estimate a rather complicated constant depending on the dimension and the power $s$.
\end{abstract}

\maketitle

\section{Introduction}

The fractional Laplace operator is an integro-differential nonlocal operator of non-integer order. 
It is defined as
\begin{align*}
\Ds u(x) & := \frac{4^s\Gamma(n/2+s)}{\pi^{n/2}\big|\Gamma(-s)\big|}
\;\pv\int_{\R^n}\frac{u(x)-u(y)}{{|x-y|}^{n+2s}}\;dy 
&  s\in(0,1),\ x\in\R^n,
\end{align*}
where ``$\pv$'' means that the integral is taken in the principal value sense.
We refer to \cites{hitchhiker,bucur-valdinoci,av,garofalo} for all its basic features. 

Here, we recall that it is naturally related to the fractional Sobolev space
\begin{align*}
H^s(\R^n)=\left\{u\in L^2(\R^n):[u]^2_{s}:=\int_{\R^n}\int_{\R^n}\frac{{|u(x)-u(y)|}^2}{{|x-y|}^{n+2s}}\;dy\;dx<\infty\right\}
\end{align*}
and, when the attention is restricted to a bounded domain $\Omega\subset\R^n$, to the space
\begin{align*}
\cH^s_0(\Omega)=\big\{u\in H^s(\R^n):u=0 \text{ in }\R^n\setminus\Omega\big\},
\end{align*} 
which is encoding a natural notion of homogeneous boundary conditions in $\R^n\setminus\Omega$:
for this reason it is sometimes also known as the \textit{restricted} fractional Laplacian.
From a functional analytic perspective, $\Ds$ is a positive self-adjoint operator on $\{u\in L^2(\R^n):u=0 \text{ in }\R^n\}$ with compact inverse.
It has therefore a discrete spectrum and the eigenvalues have finite multiplicity. In particular, the first eigenvalue, which we denote by $\lambda=\lambda(\Omega)$, is simple. It is known that the first eigenfunction $\phi \in \cH^s_0(\Omega)$ is smooth inside $\Omega$ and that it can be chosen to be strictly positive.

In this paper we partially answer a conjecture raised by Ba\~nuelos, Kulczycki, and M\'endez-Hern\'andez \cite{MR2217951}*{Conjecture 1.1}:
\begin{equation}\label{conjecture}
\text{If $n=1$ and $\Omega=(-1,1)$, then $\phi$ is concave in its support.}
\end{equation}
This has been previously established by Ba\~nuelos and Kulczycki \cite{MR2056835}*{Theorem 4.7}, for $s=\frac12$, and
by Ka\ss mann and Silvestre \cite{ks} and Ba\~nuelos and DeBlassie~\cite{MR3306696}*{Theorem 1.1}, whenever $s^{-1}\in\N$;
moreover, in general dimension and for a general bounded Lipschitz domain, \cite{MR3306696}*{Theorem 1.1} also shows that $\phi$ is superharmonic, again under the assumption $s^{-1}\in\N$.
Another related result is contained in~\cite{MR2217951}*{Theorem 1.1}, which states that $\phi$ is \textit{mid-concave} (see \cite{MR2217951}*{Definition 1.1}) on rectangles $\Omega=(-a_1,a_1)\times\cdots\times(-a_n,a_n)\subset\R^n$. 

Here, we give a computer-aided proof of \eqref{conjecture} for any $s\in(\frac12,1)$. 
\begin{theorem}\label{main theorem}
Let $s\in(\frac12,1)$ and $n=1$. Let $\phi\in \cH^s_0\big((-1,1)\big)$ denote the first eigenfunction of $\Ds$ on the interval $(-1,1)$. Then
\begin{align*}
\phi''<0\qquad\text{in }(-1,1).
\end{align*}
\end{theorem}

More generally, our approach is able to reach the following.  

\begin{theorem}\label{main theorem 2}
Let $s\in(\frac12,1)$, $2\leq n\leq 11$, and $B_1\subset \R^n$ denote the unitary ball. Let $\phi\in \cH^s_0(B_1)$ denote the first eigenfunction of $\Ds$ on $B_1$. Then
\begin{align}\label{conjecture n}
-\Delta \phi>0\qquad\text{in }B_1.
\end{align}
\end{theorem}

We believe the threshold $n\leq 11$ to be merely technical and due to a few sub-optimal estimates involved in our analysis.

Our strategy begins with a purely analytic approach to reduce $\Delta\phi$ in an integral form. 
At the core of this strategy we exploit the semigroup property of $\Ds$ in that we split
\begin{align}
-\lapl\phi &={(-\lapl)}^{1-s}\Ds\phi =
{(-\lapl)}^{1-s}\big(\lambda\phi+\1_{\R^n\setminus B_1}\Ds\phi\big) \nonumber \\
& =
\lambda{(-\lapl)}^{1-2s}\big(\lambda\phi+\1_{\R^n\setminus B_1}\Ds\phi\big)+
{(-\lapl)}^{1-s}\big(\1_{\R^n\setminus B_1}\Ds\phi\big)
& \qquad\text{in }B_1. \label{splitted}
\end{align}
Next, we write suitable integral representations for the different terms that appear: these use in a crucial way our standing assumption $s>\frac12$, so that $(-\Delta)^{1-2s}$ stands for the convolution with the fundamental solution in $\R^n$. A central role in these formulas is played by the nonlocal Poisson kernel of $B_1$
\begin{align}\label{poisson}
P_s(x,y) := \frac{\gamma(n,s)}{{|x-y|}^n}\frac{\big(1-|x|^2\big)^s}{\big(|y|^2-1\big)^s}, 
\quad \gamma(n,s):=\frac{\Gamma(\frac{n}2)}{\pi^{n/2}\Gamma(s)\,\Gamma(1-s)},
\qquad x\in B_1,\ y\in \R^n\setminus B_1.
\end{align}
The splitting of $-\lapl \phi$ is performed in Section \ref{sec:set-up}. 
A refinement of \eqref{splitted} leads to write \eqref{conjecture n} as an integral inequality not involving directly $\phi$ or $\lambda$ (see \eqref{Q} below). 
At this point we split our analysis in three different cases:
\begin{itemize}
\item For $n=1$ and $s=\frac34$ the argument can be concluded by hand, without too much of a hustle: this is done in Section \ref{logarithmic case}.
\item For $n=1$ and $s\in(\frac12,1)\setminus\{\frac34\}$ the integral quantities can be simplified a lot via estimates from below; still, the resulting inequality contains a quite complicated expression in $s$ and we therefore plot and verify it using a computer: this is done in Section \ref{1 s};
\item For $n\geq 2$ a similar approach to the previous point is taken, with the important difference that in this case hypergeometric functions make their appearance in our study: these make the analysis even more complicated and we consequently need to rely even more on numerical evaluations: this is done in Section \ref{n s}.
\end{itemize}

Many details will be deferred to Appendices, in the attempt of not breaking the flow of the exposition with technicalities. Nevertheless, we would like to mention that in Appendix \ref{bound first eigenvalue} we derive an upper bound for the first eigenvalue $\lambda$ on $B_1$ for general dimension $n$ and power $s$, while in Appendix \ref{minimum at 0} we derive from a representation formula for certain $s$-harmonic functions some symmetry and monotonicity properties which are useful in our analysis and might be of independent interest.

\subsection{Notations}
We denote by $B_r$ the $n$-dimensional ball of radius $r>0$ centered at $0$. We set
\begin{align}\label{Ftau}
F_\tau(z) & := \left\lbrace\begin{aligned}
& \kappa(n,\tau)|z|^{2\tau-n},\quad \kappa(n,\tau):=\frac{\Gamma(\frac{n}2-\tau)}{4^\tau\pi^{n/2}\big|\Gamma(\tau)\big|}, && 
\tau\in\R,\ \tau-\frac{n}2\not\in\N_0,  \\
& (-1)^{1+\tau-n/2}\frac{2^{1-2\tau}\pi^{-n/2}}{\Gamma(1+\tau-\frac{n}2)\,\Gamma(\tau)}\,|z|^{2\tau-n}\,\ln|z|, && \tau\in\R,\ \tau-\frac{n}2\in\N_0,  
\end{aligned}\right.  & z\in\R^n\setminus\{0\}.
\end{align}
Note that $F_{\tau}$ is the fundamental solution of $(-\Delta)^{\tau}$ in $\R^n$ if $\tau>0$ and for $\tau\in(-1,0)$ it is the kernel of the fractional Laplacian of order $2\tau$. For a measurable set $A\subset\R^n$, $\1_A$ denotes the characteristic function of $A$ and $A^c=\R^n\setminus A$ the complementary set of $A$.

\section{Set-up of the proof of Theorems \texorpdfstring{\ref{main theorem}}{1.1} and  \texorpdfstring{\ref{main theorem 2}}{1.2}}
\label{sec:set-up}

\subsection{Representation of the Laplacian of the eigenfunction}
In this paragraph we perform the splitting of $-\lapl\phi$ announced in \eqref{splitted}.
\begin{lemma} For $s\in(0,1)$ it holds $\Ds\phi\in L^1(\R^n)$ and we have
\begin{align}\label{dsphi}
\Ds\phi=\lambda\phi-\1_{B_1^c}(F_{-s}\ast\phi) \qquad\text{in }\R^n\setminus \partial B_1.
\end{align}
\end{lemma}
\begin{proof}
Consider $\psi\in C^\infty_c(\R^n)$. Then
\begin{align*}
\int_{\R^n}\phi\,\Ds\psi & =\int_{B_1}\phi\,\Ds\mathbb{G}_s[\Ds\psi]
=\lambda\int_{B_1}\phi\,\mathbb{G}_s[\Ds\psi]
\end{align*}
Above, we have denoted by $\mathbb{G}_s:L^2(B_1)\to\cH^s_0(B_1)$ the solution map to the Dirichlet problem
\begin{align*}
\left\lbrace\begin{aligned}
\Ds u &= f && \text{in }B_1,\ f\in L^2(B_1), \\
u &\in \cH^s_0(B_1),
\end{aligned}\right.
\end{align*}
which admits a representation in terms of the Green function
\begin{align}
\mathbb{G}_s[f](x) &= \int_{B_1}G_s(x,y)\,f(y)\;dy && \text{for }x\in B_1, \nonumber \\
G_s(x,y) &= F_s(x-y)-\int_{B_1^c}P_s(x,z)\,F_s(z-y)\;dz
&& \text{for }x,y\in B_1, \nonumber \\
P_s(x,y) &= \int_{B_1}F_{-s}(z-y)\,G_s(x,z)\;dz
&& \text{for }x\in B_1,\ y\in B_1^c,\label{pois-green}
\end{align}
see \cite{MR3461641}*{Definition 1.9} and \cite{MR3393247}*{equation (25) and Theorem 1.2}. We have then
\begin{align*}
\frac1\lambda\int_{\R^n}\phi\,\Ds\psi &=
\int_{B_1}\phi(x)\int_{B_1}F_s(x-y)\,\Ds\psi(y)\;dy\;dx \\
&\quad -\int_{B_1}\phi(x)\int_{B_1^c}P_s(x,z)\int_{B_1}F_s(z-y)\,\Ds\psi(y)\;dy\;dz\;dx \\
&=
\int_{B_1}\phi\psi
-\int_{B_1}\phi(x)\int_{B_1^c}F_s(x-y)\,\Ds\psi(y)\;dy\;dx \\
&\quad -\int_{B_1}\phi(x)\int_{B_1^c}P_s(x,z)\,\psi(z)\;dz\;dx \\
&\quad +\int_{B_1}\phi(x)\int_{B_1^c}P_s(x,z)\int_{B_1^c}F_s(z-y)\,\Ds\psi(y)\;dy\;dz\;dx \\
&=
\int_{B_1}\phi\psi-\int_{B_1}\phi(x)\int_{B_1^c}P_s(x,z)\,\psi(z)\;dz\;dx \\
&\quad +\int_{B_1}\phi(x)\int_{B_1^c}\bigg(\int_{B_1^c}P_s(x,z)\,F_s(z-y)\;dz-F_s(x-y)\bigg)\Ds\psi(y)\;dy\;dx. 
\end{align*}
As it holds (see \cite{MR0350027}*{equation (1.6.12')})
\begin{align*}
\int_{B_1^c}P_s(x,z)F_s(z-y)\;dz=F_s(x-y)
\qquad\text{for }x\in B_1,\ y\in B_1^c,
\end{align*}
then
\begin{align}
\int_{\R^n}\phi\,\Ds\psi
& =\lambda\int_{B_1}\phi\,\psi-\lambda\int_{B_1}\phi(x)\int_{B_1^c}P_s(x,z)\,\psi(z)\;dz\;dx \nonumber \\
& =\lambda\int_{B_1}\phi\,\psi-\int_{B_1^c}\psi(z)\int_{B_1}P_s(x,z)\,\lambda\phi(x)\;dx\;dz \nonumber \\
& =\lambda\int_{B_1}\phi\,\psi-\int_{B_1^c}\psi\,(F_{-s}\ast\mathbb{G}_s[\lambda\phi]) 
  =\lambda\int_{B_1}\phi\,\psi-\int_{B_1^c}\psi\,(F_{-s}\ast\phi) \label{64872364}
\end{align}
where we have used \eqref{pois-green}.
The stated equality \eqref{dsphi} holds also pointwisely in view of \cite{MR2270163}*{Proposition 2.4}.
\end{proof}

\begin{proposition} For $s\in(\frac12,1)$ it holds
\begin{align}
-\Delta\phi &= \lambda^2F_{2s-1}\ast\phi+\big(F_{s-1}-\lambda F_{2s-1}\big)\ast\big[\1_{B_1^c}(F_{-s}\ast\phi)\big] & \label{laplacian1}\\
&= \lambda^2F_{2s-1}\ast\phi+\lambda\big(F_{s-1}-\lambda F_{2s-1}\big)\ast\mathbf{P}_s[\phi] & \text{in }B_1, \label{laplacian2}
\end{align}
where
\begin{align*}
\mathbf{P}_s[\phi](y)=\1_{B_1^c}(y)\int_{B_1}P_s(x,y)\,\phi(x)\;dx
\qquad\text{for }y\in\R^n.
\end{align*}
\end{proposition}
\begin{proof}
Starting from the last lemma, we compute (all equalities hold only in $B_1$)
\begin{align*}
-\Delta\phi={(-\Delta)}^{1-s}\Ds\phi={(-\Delta)}^{1-s}\big[\lambda\phi-\1_{B_1^c}(F_{-s}\ast\phi)\big]
=\lambda{(-\Delta)}^{1-s}\phi+F_{s-1}\ast[\1_{B_1^c}(F_{-s}\ast\phi)].
\end{align*}
Note that
\begin{align*}
{(-\Delta)}^{1-s}\phi=F_{2s-1}\ast\Ds\phi=F_{2s-1}\ast\big[\lambda\phi-\1_{B_1^c}(F_{-s}\ast\phi)\big]
\end{align*}
thus \eqref{laplacian1} follows.
In expanded form, \eqref{laplacian1} reads, for $x\in B_1$,
\begin{align*}
-\Delta\phi(x)
& =\lambda^2\int_{B_1}F_{2s-1}(x-y)\,\phi(y)\;dy  \\
& \qquad +\int_{B_1^c}\big(F_{s-1}(x-z)-\lambda F_{2s-1}(x-z)\big)\int_{B_1}F_{-s}(z-y)\,\phi(y)\;dy\;dz \\
& =\int_{B_1}\phi(y)\bigg[\lambda^2F_{2s-1}(x-y)+\int_{B_1^c}\big(F_{s-1}(x-z)-\lambda F_{2s-1}(x-z)\big)F_{-s}(z-y)\;dz\bigg]\;dy.
\end{align*}
Also, identity
\begin{align*}
(F_{-s}\ast\phi)(z)=\lambda\int_{B_1}P_s(y,z)\,\phi(y)\;dy
\qquad\text{for }z\in B_1^c,
\end{align*}
holds (we have already used this one in \eqref{64872364} exploiting \eqref{pois-green}). 
In expanded form, \eqref{laplacian2} reads
\begin{align*}
-\Delta\phi(x)
& =\lambda^2\int_{B_1}F_{2s-1}(x-y)\,\phi(y)\;dy  \\
& \qquad +\lambda\int_{B_1^c}\big(F_{s-1}(x-z)-\lambda F_{2s-1}(x-z)\big)\int_{B_1}P_s(y,z)\,\phi(y)\;dy\;dz \\
& =\lambda\int_{B_1}\phi(y)\bigg[\lambda F_{2s-1}(x-y)+\int_{B_1^c}P_s(y,z)\big(F_{s-1}(x-z)-\lambda F_{2s-1}(x-z)\big)\;dz\bigg]\;dy.
\end{align*}
\end{proof}

We know that, by uniqueness, $\phi$ and $-\Delta\phi$ are radial, so that for any $x\in B_1$ fixed
\begin{align*}
-\Delta\phi(x)=-\frac1{\big|\partial B_{|x|}\big|}\int_{\partial B_{|x|}}\Delta\phi(\theta)\;d\theta.
\end{align*}
Keeping this in mind, we define
\begin{align}\label{defi:j-tau}
J_\tau(x;y) := 
\left\lbrace\begin{aligned}
& \frac1{\big|\partial B_{|x|}\big|}\int_{\partial B_{|x|}}F_\tau(\theta-y)\;d\theta
& & \text{for }n\geq 2,\,x,y\in\R^n,\,x\neq y  \\
& \frac{F_\tau(x-y)+F_\tau(x+y)}{2} 
& & \text{for }n=1,\,x,y\in\R,\,x\neq y .
\end{aligned}\right.
\end{align}
Using \eqref{laplacian2}, we then write
\begin{align*}
-\Delta\phi(x)
& =\lambda\int_{B_1}\phi(y)\bigg[\lambda J_{2s-1}(x;y)+\int_{B_1^c}P_s(y,z)\big(J_{s-1}(x;z)-\lambda J_{2s-1}(x;z)\big)\;dz\bigg]\;dy,
\quad x\in B_1.
\end{align*}
Since
\begin{multline*}
\lambda J_{2s-1}(x;y)+\int_{B_1^c}P_s(y,z)\big(J_{s-1}(x;z)-\lambda J_{2s-1}(x;z)\big)\;dz=\\
=\lambda \bigg( J_{2s-1}(x;y)-\int_{B_1^c}P_s(y,z) \, J_{2s-1}(x;z)\;dz\bigg)+\int_{B_1^c}P_s(y,z) \, J_{s-1}(x;z)\;dz,
\end{multline*}
the positivity of $-\Delta\phi$ follows once we show
\begin{align*}
\lambda \bigg( J_{2s-1}(x;y)-\int_{B_1^c}P_s(y,z) \, J_{2s-1}(x;z)\;dz\bigg)+\int_{B_1^c}P_s(y,z) \, J_{s-1}(x;z)\;dz&\geq 0,
\qquad x,y\in B_1.
\end{align*}
As the second addend is clearly positive in the above inequality, we may replace $\lambda$ with a larger constant (see Appendix \ref{bound first eigenvalue}, equation \eqref{new bound on lambda})
\begin{align}\label{bigger lambda}
\lambda\leq \Lambda(n,s):=\frac{4^s\Gamma(1+s)^2\,\Gamma(1+2s+\frac{n}{2})}{(s+\frac{n}{2})\,\Gamma(\frac{n}2)\,\Gamma(1+2s)}
\end{align}
and it is then enough to show
\begin{multline}\label{Q}
J_{2s-1}(x;y)-\int_{B_1^c}P_s(y,z) \, J_{2s-1}(x;z)\;dz+\frac1{\Lambda(n,s)}\int_{B_1^c}P_s(y,z) \, J_{s-1}(x;z)\;dz\geq 0, \\
s\in\Big(\frac12,1\Big),\ x,y\in B_1.
\end{multline}

Note that $s-1-\frac{n}{2}<0$ for all $n\in\N$ and $s\in(\frac12,1)$, but $2s-1-\frac{n}{2}\in \N_0$ if and only if $s=\frac{3}{4}$ and $n=1$. Hence, this case differs strongly from the other cases as an effect of definitions \eqref{Ftau} and \eqref{defi:j-tau}. 
We begin with some general estimates to simplify \eqref{Q}.

\subsection{Reformulation of \texorpdfstring{\eqref{Q}}{(Q)}}
The first step is to note that the left-hand side of \eqref{Q} actually depends only on $|x|$ and $|y|$.
\begin{lemma}
For any $n\in \N$, $\tau\in \R$, and $s\in(0,1)$ we have
\begin{align*}
J_{\tau}(x;y)&=J_{\tau}(|x|e_1,|y|e_1)
&& x,y\in \R^n,\, x\neq y,\\
\int_{B_1^c}P_s(y,z)J_{\tau}(x;z)\ dz&=\int_{B_1^c}P_s(|y|e_1,z)\,J_{\tau}(|x|e_1;z)\;dz
&& x,y\in B_1,\, x\neq y.
\end{align*}
\end{lemma}
\begin{proof}
First note that $F_{\tau}(x)=F_{\tau}(|x|e_1)$ for any $x\in \R^n\setminus\{0\}$ by definition \eqref{Ftau}. To see the statement for $J_{\tau}$, note that this is obviously true for $n=1$ from definition \eqref{defi:j-tau} and for $n>1$ we have by a rotation for $x\in B_1$, $y\in \R^n\setminus\{x\}$
\begin{align*}
& J_{\tau}(x;y)=\frac{1}{|\partial B_{|x|}|}\int_{\partial B_{|x|}}F_{\tau}\big(|\theta-y|e_1\big)\; d\theta=\frac{1}{|\partial B_{|x|}|}\int_{\partial B_{|x|}}F_{\tau}\Big(\sqrt{1-2\theta\cdot y+|y|^2}\,e_1\Big)\;d\theta\\
&=\frac{1}{|\partial B_{|x|}|}\int_{\partial B_{|x|}}F_{\tau}\Big(\sqrt{1-2|y|\nu_1+|y|^2}e_1)\; d\nu=\frac{1}{|\partial B_{|x|}|}\int_{\partial B_{|x|}}F_{\tau}\big(\nu-|y|e_1\big)\;d\nu=J_{\tau}(|x|e_1;|y|e_1).
\end{align*}
Similarly, we have with a rotation
\begin{align*}
\int_{B_1^c}P_s(y,z)\,J_{\tau}(x;z)\;dz&=\frac{\Gamma(\frac{n}{2})}{\pi^{\frac{n}{2}}\Gamma(s)\Gamma(1-s)}\int_{B_1^c}\frac{(1-|y|^2)^s}{(|z|^2-1)^s|y-z|^n}\,J_{\tau}(|x|e_1;|z|e_1)\;dz\\
&=\frac{\Gamma(\frac{n}{2})}{\pi^{\frac{n}{2}}\Gamma(s)\Gamma(1-s)}\int_{B_1^c}\frac{(1-|y|^2)^s}{(|z|^2-1)^s\big(|y|^2-2y\cdot z+|z|^2\big)^{\frac{n}{2}}}\,J_{\tau}(|x|e_1;|z|e_1)\;dz\\
&=\frac{\Gamma(\frac{n}{2})}{\pi^{\frac{n}{2}}\Gamma(s)\Gamma(1-s)}\int_{B_1^c}\frac{(1-|y|^2)^s}{(|v|^2-1)^s\big(|y|^2-2|y|v_1+|v|^2\big)^{\frac{n}{2}}}\,J_{\tau}(|x|e_1;|v|e_1)\;dv\\
&=\frac{\Gamma(\frac{n}{2})}{\pi^{\frac{n}{2}}\Gamma(s)\Gamma(1-s)}\int_{B_1^c}\frac{(1-|y|^2)^s}{(|v|^2-1)^s\big||y|e_1-v\big|^n}\,J_{\tau}(|x|e_1;|v|e_1)\;dv\\
&=\int_{B_1^c}P_s(|y|e_1,v)\,J_{\tau}(|x|e_1;v)\;dv.
\end{align*}
\end{proof}
In view of the last lemma \eqref{Q} reduces to 
\begin{multline*}
J_{2s-1}(|x|e_1;|y|e_1)-\int_{B_1^c}P_s(|y|e_1,z) \, J_{2s-1}(|x|e_1;z)\;dz+\frac1{\Lambda(n,s)}\int_{B_1^c}P_s(|y|e_1,z) \, J_{s-1}(|x|e_1;z)\;dz\geq 0, \\
s\in\Big(\frac12,1\Big),\ x,y\in B_1.
\end{multline*}

\begin{lemma}\label{lem:integral computation}
Let $n\geq 2$. For $\varepsilon,r>0$ and $\alpha\in\R$ it holds
\begin{align*}
\int_{\partial B_r}\frac{dx}{\big|x-\eps e_1\big|^{2\alpha}}=\frac{2\pi^{\frac{n-1}{2}}}{\Gamma(\frac{n-1}{2})}\,r^{n-1}\int_{-1}^1\frac{(1-t^2)^{\frac{n-3}{2}}}{\big(r^2+\eps^2-2\eps r t\big)^{\alpha}}\;dt
\end{align*}
In case $r=\eps$, we additionally require $2\alpha<n-1$. In particular, we have for $\tau-\frac{n}{2}\notin\N_0$ 
\[
J_{\tau}(x;y)=\frac{2\pi^{\frac{n-1}2}}{\Gamma(\frac{n-1}2)}\,\kappa(n,\tau)|x|^{n-1}
\int_{-1}^1\frac{(1-t^2)^{\frac{n-3}{2}}}{\big(|x|^2+|y|^2-2|x||y| t\big)^{\frac{n}{2}-\tau}}\;dt\qquad x,y\in\R^n,
\]
where we require additionally $|x|\neq |y|$ if $\tau>\frac12$.
\end{lemma}
\begin{proof}
Via an explicit calculation
\begin{multline*}
\int_{\partial B_r}\frac{dx}{\big|x-\eps e_1\big|^{2\alpha}} =
r^{n-1}\int_{\partial B_1}\frac{dy}{\big|ry-\eps e_1\big|^{2\alpha}}
=r^{n-1}\int_{\partial B_1} \frac{dy}{\big(r^2-2r\eps y_1+\eps^2\big)^{\alpha}}=\\
=\frac{2\pi^{\frac{n-1}{2}}}{\Gamma(\frac{n-1}{2})}\,
r^{n-1}\int_{-1}^1\frac{(1-t^2)^{\frac{n-3}{2}}}{\big(r^2+\eps^2-2\eps r t\big)^{\alpha}}\;dt.
\end{multline*}
The last part follows by setting $r=|x|$, $\varepsilon=|y|$, and $\alpha=\frac{n}{2}-\tau$.
\end{proof}

\begin{lemma}\label{lemma-monotonicity1}
For $x\in[0,1)$ and $s\in(\frac12,1)$ let 
\begin{align*}
f_1:[0,\infty)\setminus\{x\} & \longrightarrow\R &&& f_2:(1,\infty) & \longrightarrow \R \\
y &\longmapsto J_{2s-1}(xe_1,ye_1) &&& z & \longmapsto J_{s-1}(xe_1,ze_1).
\end{align*}
The following holds.
\begin{enumerate}
\item If $n=1$, then $f_1$ is (\textit{cf}. Figure \ref{fig:J}):
\begin{enumerate}
\item positive for $\frac12<s<\frac34$,
\item negative for $\frac34<s<1$,
\item decreasing in $(x,\infty)$,
\item convex in $(0,x)$.
\end{enumerate}
\item If $n=1$, $s=\frac{3}{4}$, then $f_1(y)=-\frac{1}{2\pi}\ln(|x^2-y^2|)$, so $f_1$ satisfies (c) and (d) in 1. Moreover, $f_1>0$ in $[0,\sqrt{1+x^2})\setminus\{x\}$ and $f_1<0$ in $(\sqrt{1+x^2},\infty)$.
\item $f_2$ is positive and decreasing.
\item For $n\geq 2$ the function $f_1$ is positive and satisfies (c) and (d) in (1).
\end{enumerate}
\end{lemma}

\begin{proof} 
Claims \textit{(1)} and \textit{(2)} follow easily from the definition. 
\begin{figure}[ht]
\begin{tikzpicture}
	\draw[very thick,->] (-.5\textwidth,0pt) -- (-.05\textwidth,0pt);
	\node at (-.05\textwidth+5pt,0) {\footnotesize $y$};
	\draw (-.45\textwidth,-5pt) -- (-.45\textwidth,.2\textwidth);
	\node at (-.45\textwidth,-10pt) {\footnotesize $-1$};
	\draw (-.1\textwidth,-5pt) -- (-.1\textwidth,.2\textwidth);
	\node at (-.1\textwidth,-10pt) {\footnotesize $1$};
	\filldraw (-.275\textwidth,0) circle (1.3pt);
	\node at (-.275\textwidth,-10pt) {\footnotesize $0$};
	\draw[densely dotted] (-.36\textwidth,-3pt) -- (-.36\textwidth,.2\textwidth);
	\node at (-.36\textwidth,-10pt) {\footnotesize $-x$};
	\draw[densely dotted] (-.19\textwidth,-3pt) -- (-.19\textwidth,.2\textwidth);
	\node at (-.19\textwidth,-10pt) {\footnotesize $x$};
	\draw [thick] plot [smooth,tension=.75] coordinates {(-.5\textwidth,13pt) (-.45\textwidth,20pt) (-.38\textwidth,45pt) (-.36\textwidth-2pt,.2\textwidth)};
	\draw [thick] plot [smooth,tension=2] coordinates {(-.36\textwidth+2pt,.2\textwidth) (-.275\textwidth,30pt) (-.19\textwidth-2pt,.2\textwidth)};
	\draw [densely dotted] (-.45\textwidth,20pt) -- (-.1\textwidth,20pt);
	\draw [thick] plot [smooth,tension=.75] coordinates {(-.05\textwidth,13pt) (-.1\textwidth,20pt) (-.17\textwidth,45pt) (-.19\textwidth+2pt,.2\textwidth)};
	\draw [densely dashed] plot [smooth,tension=1.9] coordinates {(-.45\textwidth,20pt) (-.275\textwidth,8pt) (-.1\textwidth,20pt)};
\end{tikzpicture}
\hfill
\begin{tikzpicture}
	\draw[very thick,->] (-.5\textwidth,0pt) -- (-.05\textwidth,0pt);
	\node at (-.05\textwidth+5pt,0) {\footnotesize $y$};
	\draw (-.45\textwidth,5pt) -- (-.45\textwidth,-.2\textwidth);
	\node at (-.45\textwidth,10pt) {\footnotesize $-1$};
	\draw (-.1\textwidth,5pt) -- (-.1\textwidth,-.2\textwidth);
	\node at (-.1\textwidth,10pt) {\footnotesize $1$};
	\draw[densely dotted] (-.36\textwidth,3pt) -- (-.36\textwidth,-.2\textwidth);
	\node at (-.36\textwidth,10pt) {\footnotesize $-x$};
	\draw[densely dotted] (-.19\textwidth,3pt) -- (-.19\textwidth,-.2\textwidth);
	\node at (-.19\textwidth,10pt) {\footnotesize $x$};
	\filldraw (-.275\textwidth,0) circle (1.3pt);
	\node at (-.275\textwidth,10pt) {\footnotesize $0$};
	\draw [thick] plot [smooth,tension=.81] coordinates {(-.5\textwidth,-60pt) (-.45\textwidth,-55pt) (-.38\textwidth,-37pt) (-.36\textwidth,-10pt)};
	\draw [thick] plot [smooth,tension=2] coordinates {(-.36\textwidth,-10pt) (-.275\textwidth,-40pt) (-.19\textwidth,-10pt)};
	\draw [densely dotted] (-.45\textwidth,-55pt) -- (-.1\textwidth,-55pt);
	\draw [thick] plot [smooth,tension=.81] coordinates {(-.05\textwidth,-60pt) (-.1\textwidth,-55pt) (-.17\textwidth,-37pt) (-.19\textwidth,-10pt)};
	\draw [densely dashed] plot [smooth,tension=1.9] coordinates {(-.45\textwidth,-55pt) (-.275\textwidth,-78pt) (-.1\textwidth,-55pt)};
\end{tikzpicture}
\caption{On the left, a qualitative graph of $f_1$ and of its $s$-harmonic extension as entailed by the Poisson integral in \eqref{Q} in $(-1,1)$ (dashed line) for $\frac12<s<\frac34$; on the right, the analogue picture for $\frac34<s<1$.}
\label{fig:J}
\end{figure}
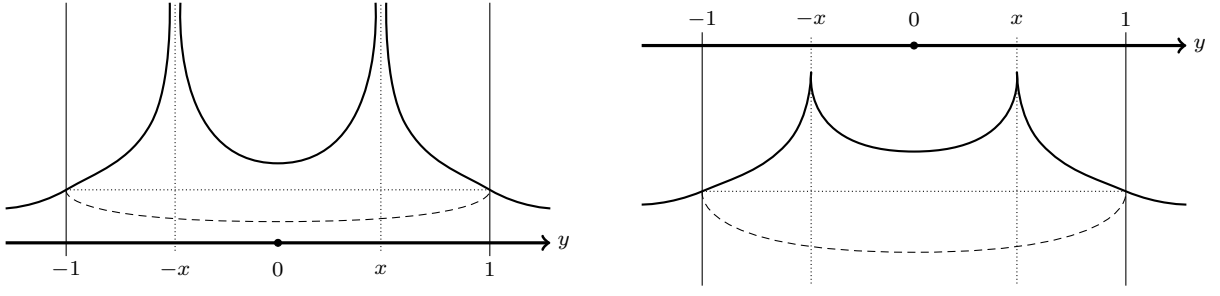

For \textit{(3)} note that by definition and Lemma \ref{lem:integral computation} we have
\[
f_2(z)=\left\{\begin{aligned}
&\frac{2^{1-2s}\Gamma(\frac32-s)(1-s)}{\sqrt{\pi}\,\Gamma(s)}\Big(|x-z|^{2s-3}+|x+z|^{2s-3}\Big)&& n=1;\\
& \frac{(1-s)\Gamma(\frac{n}{2}-s+1)}{2^{2s-3}\sqrt{\pi}\Gamma(s)\Gamma(\frac{n-1}{2})}\,x^{n-1}
\int_{-1}^1\frac{(1-t^2)^{\frac{n-3}{2}}}{(x^2+z^2-2xz t)^{\frac{n}{2}-s+1}}\;dt && n>1,
\end{aligned}\right.
\]
from where it easily follows that we have $f_2'(z)<0<f_2(z)$ for $z\geq 1$.

For \textit{(4)} note that by Lemma \ref{lem:integral computation} we have
\[
f_1(y)=\frac{\Gamma(\frac{n}{2}-2s+1)}{2^{4s-3}\sqrt{\pi}\Gamma(2s-1)\Gamma(\frac{n-1}{2})} \int_{-1}^1\frac{(1-t^2)^{\frac{n-3}{2}}}{(x^2+y^2-2xy t)^{\frac{2+n}{2}-2s}}\;dt
\]
and from here the statement follows again.
\end{proof}

\begin{lemma}\label{lemma-monotonicity2}
For any $s\in(\frac12,1)$, $n\in \N$, and $x,y\in[0,1)$ it holds
\begin{align*}
J_{2s-1}(xe_1;ye_1)-\int_{B_1^c}P_s(ye_1,z) \, J_{2s-1}(xe_1;z)\;dz&\geq J_{2s-1}(xe_1;0)-J_{2s-1}(xe_1,e_1) \\
\int_{B_1^c}P_s(ye_1,z) \, J_{s-1}(xe_1;z)\;dz&\geq \int_{B_1^c}P_s(0,z) \, J_{s-1}(xe_1;z)\;dz.
\end{align*}
\end{lemma}
\begin{proof}
These statements follow from Proposition \ref{prop:s-harm} in combination with Remarks \ref{rmk:increasing}, \ref{rmk:sign-changing}, and Lemma \ref{lemma-monotonicity1}.
\end{proof}

In view of Lemma \ref{lemma-monotonicity2} it follows that \eqref{Q} holds once we show the nonnegativity of the function 
\begin{equation}\label{Q-easier}
\begin{split}
[0,1) & \longrightarrow\R \\
x & \longmapsto J_{2s-1}(xe_1;0)-J_{2s-1}(xe_1;e_1)+\frac{1}{\Lambda(n,s)}\int_{B_1^c}P_s(0,z) \, J_{s-1}(xe_1;z)\;dz
\end{split}
\end{equation}

Noting that the last addend is positive, the positivity of the above immediately follows for those $x\in[0,1)$ for which one has
$J_{2s-1}(xe_1;0)\geq J_{2s-1}(xe_1;e_1)$, which is what we study next. 

\begin{lemma}\label{lem:midconcavity}
Let $s\in(\frac12,1)$. Then there exists $x_*(n,s)\in(\frac12,1)$ such that
\[
J_{2s-1}(xe_1;0)\geq J_{2s-1}(xe_1;e_1)
\qquad\text{for all }x\in[0,x_*(n,s)].
\]
More precisely, one can take
\[
x_*(n,s)=\left\{\begin{aligned} & \frac23 && \text{for }\, n=1 \text{ and } s=\frac34;\\
& \frac1{1+\big(2-2^{4s-2-n}\big)^{\frac1{4s-2-n}}} && \text{for }n=1 \text{ and } s\neq \frac34,\text{ or } n\geq 2.
\end{aligned}\right.
\]
\end{lemma}

\begin{remark}
\hspace{1em}
\begin{enumerate}
\item Note here that $x_*(1,s)\geq \frac35$. The statement of Lemma \ref{lem:midconcavity} hence gives an alternative proof to the mid-concavity as shown in~\cite{MR2217951}*{Theorem 1.1}, although just in dimension $n=1$.
\item Similarly to the previous point, $x_*(n,s)>\frac12$ for any $n\in \N$. The statement of Lemma \ref{lem:midconcavity} gives therefore super-harmonicity in the ball $B_{1/2}$ in any dimension.
\item To show Theorems \ref{main theorem} and \ref{main theorem 2}, it is in view of Lemma \ref{lem:midconcavity} enough to show the positivity of \eqref{Q-easier} for $x\in(x_*(n,s),1)$ and $s\in(\frac12,1)$.
\end{enumerate}
\end{remark}

\begin{proof}[Proof of Lemma \ref{lem:midconcavity}]
If $n=1$, $s=\frac34$, then 
\[
J_{2s-1}(x;y)=J_{\frac12}(x;y)=-\frac{1}{2\pi}\ln(|x^2-y^2|)\quad \text{for $x\in[0,1)$, $y\geq 0$, with $x\neq y$.}
\]
Hence,
\[
0\leq J_{2s-1}(x;0)-J_{2s-1}(x;1)=-\frac1{2\pi}\ln(x^2)+\frac{1}{\pi}\ln(1-x^2)=\frac1{2\pi}\ln\Big(\frac{1}{x^2}-1\Big)
\]
if and only if $x\leq \frac{\sqrt{2}}{2}$ and clearly $x_*(1,\frac34)=\frac23<\frac{\sqrt{2}}{2}$.\\
If $n=1$ and $s\neq \frac34$, it follows that $J_{2s-1}(x;0)\geq J_{2s-1}(x;1)$ holds for those $x$'s, where we have 
\begin{align*}
2\Gamma(3/2-2s)x^{4s-3}\geq\Gamma(3/2-2s)\Big({(1-x)}^{4s-3}+{(1+x)}^{4s-3}\Big).
\end{align*}
Noting that $\Gamma(3/2-2s)$ changes its sign at $s=\frac34$, the claim amounts to checking
\begin{align*}
\begin{aligned}
& 2{x}^{4s-3}\geq{(1-x)}^{4s-3}+{(1+x)}^{4s-3}
&& \text{for }s\in\Big(\frac12,\frac34\Big), \\
& 2{x}^{4s-3}\leq{(1-x)}^{4s-3}+{(1+x)}^{4s-3}
&& \text{for }s\in\Big(\frac34,1\Big).
\end{aligned}
\end{align*}
A sufficient condition, is then given by
\begin{align*}
\left\lbrace\begin{aligned}
\big(2-2^{4s-3}\big){x}^{4s-3} &\geq {(1-x)}^{4s-3} \\
2^{4s-3}{x}^{4s-3} &\geq {(1+x)}^{4s-3} \\
s &\in\Big(\frac12,\frac34\Big)
\end{aligned}\right.
\qquad\text{or}\qquad
\left\lbrace\begin{aligned}
\big(2-2^{4s-3}\big){x}^{4s-3} &\leq {(1-x)}^{4s-3} \\
2^{4s-3}{x}^{4s-3} &\leq {(1+x)}^{4s-3} \\
s &\in\Big(\frac34,1\Big)
\end{aligned}\right.
\end{align*}
which are both equivalent to
\begin{align*}
\left\lbrace\begin{aligned}
\big(2-2^{4s-3}\big)^{\frac1{4s-3}}x &\leq 1-x, \\
2x &\leq 1+x.
\end{aligned}\right.
\end{align*}
This gives $J_{2s-1}(x,0)\geq J_{2s-1}(x,1)$ for all $x$ such that
\begin{align*}
x \leq \frac1{1+\big(2-2^{4s-3}\big)^{\frac1{4s-3}}}=:x_*(1,s),
\qquad s\in\Big(\frac12,1\Big)\setminus\Big\{\frac34\Big\}.
\end{align*}
It can be easily verified that $x_*(1,\cdot):(\frac12,1)\setminus\{\frac34\}\to\R$ satisfies
\begin{align*}
& \lim_{s\downarrow1/2}x_*(1,s)=\frac35, &
& \lim_{s\uparrow1}x_*(1,s)=1, 
& \lim_{s\to3/4}x_*(1,s)=\frac23,\\
& \frac{d}{ds}x_*(1,s)>0, && \lim_{s\uparrow 1}\frac{d}{ds}x_*(1,s)=8\ln2,
& x_*(1,s)\in\Big(\frac35,1\Big).
\end{align*}

\begin{figure}
\includegraphics[width=0.5\columnwidth]{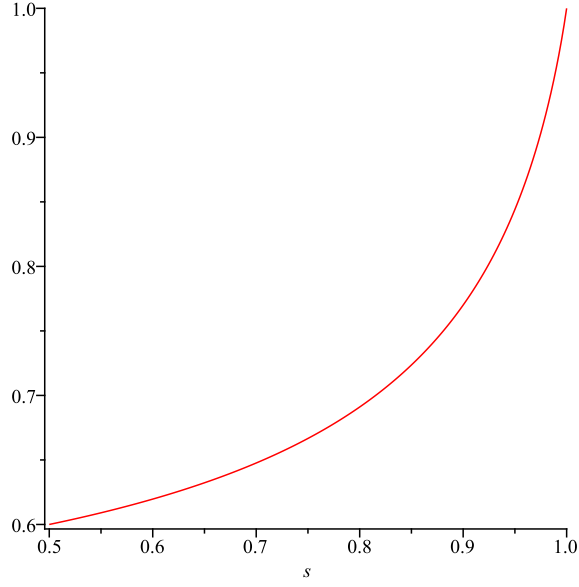}
\caption{A graph of $x_*(1,s)$, $\frac12<s<1$.}
\end{figure}

Finally, let $n\geq 2$. By Lemma \ref{lem:integral computation} we have for $x,y\in[0,1)$
\begin{equation}
\label{J-2s-1-specialcase}
J_{2s-1}(xe_1;ye_1)=\frac{\Gamma(\frac{n}{2}+1-2s)x^{n-1}}{4^{2s-1}\sqrt{\pi}\Gamma(2s-1)\Gamma(\frac{n-1}{2})} j(x,y),
\end{equation}
where we put
\[
j:[0,1]\times[0,1]\to[0,\infty],\quad j(a,b)=\int_{-1}^1\frac{(1-t^2)^{\frac{n-3}{2}}}{(a^2+b^2-2ab t)^{\frac{n}{2}+1-2s}}dt.
\]
Note here, that for $a>0$
\begin{align*}
j(a,a)&=(2a)^{4s-2-n}\int_{-1}^1(1+t)^{\frac{n-3}{2}}(1-t)^{2s-\frac{5}{2}}\ dt\ \left\{ \begin{aligned} &<\infty&& \text{for }s>\frac34;\\
&=\infty && \text{for }s\leq\frac34, \end{aligned}\right.\\
j(a,0)&=a^{4s-2-n}\int_{-1}^1(1-t^2)^{\frac{n-3}{2}}\ dt=a^{4s-2-n}\frac{\Gamma(\frac{n-1}{2})\sqrt{\pi}}{\Gamma(\frac{n}{2})},\quad\text{and}\quad\\
j(a,0)-j(a,1)&=\int_{-1}^1 (1-t^2)^{\frac{n-3}{2}}\Big(a^{4s-2-n}-\frac{1}{(a^2+1-2a t)^{\frac{n}{2}+1-2s}}\Big)dt.
\end{align*}
Since
\[
a^{4s-2-n}\geq\frac{1}{(a^2+1-2a t)^{\frac{n}{2}+1-2s}}\quad\Leftrightarrow\quad a^2+1-2a t> a^2\quad\Leftrightarrow\quad 1> 2at
\]
it follows that we have $j(a,0)-j(a,1)\geq0$ for $a\in(0,\frac12]$. To be more precise on the estimate, next note that it also follows for $a\in(\frac12,1)$
\begin{align*}
j(a,0)-j(a,1)&=  \int_{-1}^{\frac{1}{2a}} (1-t^2)^{\frac{n-3}{2}}\Big(a^{4s-2-n}-\frac{1}{(a^2+1-2a t)^{\frac{n}{2}+1-2s}}\Big)dt\\
&\qquad\qquad\qquad -\int_{\frac{1}{2a}}^1 (1-t^2)^{\frac{n-3}{2}}\Big(\frac{1}{(a^2+1-2a t)^{\frac{n}{2}+1-2s}}-a^{4s-2-n}\Big)dt\\
&\geq \Big(a^{4s-2-n}-\frac{1}{(a^2+1+2a)^{\frac{n}{2}+1-2s}}\Big)\int_{0}^{1} (1-t^2)^{\frac{n-3}{2}}\ dt\\
&\qquad\qquad\qquad -\Big(\frac{1}{(a^2+1-2a)^{\frac{n}{2}+1-2s}}-a^{4s-2-n}\Big)\int_{0}^{1} (1-t^2)^{\frac{n-3}{2}}\ dt\\
&=\frac{\Gamma(\frac{n-1}{2})\sqrt{\pi}}{2\Gamma(\frac{n}{2})}\Big( 2a^{4s-2-n}-(a+1)^{4s-2-n}-(1-a)^{4s-2-n}\Big).
\end{align*}
Arguing as in the case $n=1$, $s\neq \frac34$, we see that $j(a,0)-j(a,1)\geq 0$ holds for all $a\in(0,x_*(n,s)]$ with the claimed $x_*(n,s)$. With \eqref{J-2s-1-specialcase}, this concludes the proof.
\end{proof}

The analysis carried out so far allows us to reduce the condition in \eqref{Q} to the following stronger one:
\begin{multline}\label{Qbis}
J_{2s-1}(xe_1;0)-J_{2s-1}(xe_1;1)+\frac1{\Lambda(n,s)}\int_{B_1^c}P_s(0,z) \, J_{s-1}(xe_1;z)\;dz\geq 0, \\
s\in\Big(\frac12,1\Big),\ x_*(n,s)<x<1.
\end{multline}

Before we turn to the study of \eqref{Qbis}, let us prove a technical lemma which will be needed in the following.

\begin{lemma}\label{lem:estimate integral}
For $x\in(0,1)$ and $s\in(\frac12,1)$ it holds
\[
\int_1^\infty\frac{(z-x)^{2s-3}+(z+x)^{2s-3}}{z\,\big(z^2-1\big)^s}\;dz
\geq
\frac{\Gamma(1-s)\sqrt\pi}{\Gamma(\frac32-s)}\big(1-x^2\big)^{s-2}\big(2x-1\big).
\]
\end{lemma}
\begin{proof}
With the change of variables $z=\frac1t$, we write
\begin{align*}
\int_1^\infty\frac{(z-x)^{2s-3}+(z+x)^{2s-3}}{z\,\big(z^2-1\big)^s}\;dz
& =\int_0^1\frac{(1-xt)^{2s-3}+(1+xt)^{2s-3}}{\big(1-t^2\big)^s}\,t^2\;dt
=\int_{-1}^1\frac{(1-xt)^{2s-3}}{\big(1-t^2\big)^s}\,t^2\;dt \\
&=\int_{-1}^1\frac{(1-xt)^{2s-3}}{\big(1-t^2\big)^s}\;dt 
-\int_{-1}^1\frac{(1-xt)^{2s-3}}{\big(1-t^2\big)^{s-1}}\;dt.
\end{align*}
We integrate the second integral in the above expression by parts, obtaining
\begin{align*}
\int_{-1}^1(1-xt)^{2s-3}\big(1-t^2\big)^{1-s}\;dt=
\frac1x\int_{-1}^1(1-xt)^{2s-2}\big(1-t^2\big)^{-s}t\;dt.
\end{align*}
Via another change of variable, namely $t=2\tau-1$, we obtain
\begin{align*}
\int_{-1}^1\frac{(1-xt)^{2s-3}}{\big(1-t^2\big)^s}\;dt
&=
2^{1-2s}\int_0^1\frac{(1+x-2x\tau)^{2s-3}}{\tau^s\,(1-\tau)^s}\;d\tau \\
&=
\frac{2^{1-2s}\Gamma(1-s)^2}{\Gamma(2-2s)}(1+x)^{2s-3}\hf\Big(3-2s,1-s;2-2s\Big|\frac{2x}{1+x}\Big) \\
&=
\frac{2^{1-2s}\Gamma(1-s)^2}{\Gamma(2-2s)}(1+x)^{2s-3}\Big(1-\frac{2x}{1+x}\Big)^{s-2}\Big(1-\frac{x}{1+x}\Big) \\
&=
\frac{2^{1-2s}\Gamma(1-s)^2}{\Gamma(2-2s)}\big(1-x^2\big)^{s-2}
\end{align*}
(where we have used \eqref{hyper def} and \eqref{hyper identity 2}) and, similarly,
\begin{align*}
\int_{-1}^1\frac{(1-xt)^{2s-2}}{\big(1-t^2\big)^s}\,t\;dt
&= 
\int_{-1}^1\frac{(1-xt)^{2s-2}}{\big(1-t^2\big)^s}\;dt-\int_{-1}^1\frac{(1-xt)^{2s-2}}{(1+t)^s}\,(1-t)^{1-s}\;dt \\
&=
2^{1-2s}\int_0^1\frac{(1+x-2x\tau)^{2s-2}}{\tau^s\,(1-\tau)^s}\;d\tau
-2^{2-2s}\int_0^1\frac{(1+x-2x\tau)^{2s-2}}{\tau^s}\,(1-\tau)^{1-s}\;d\tau \\
&\leq
\frac{2^{1-2s}\Gamma(1-s)^2}{\Gamma(2-2s)}(1+x)^{2s-2}\hf\Big(2-2s,1-s;2-2s\Big|\frac{2x}{1+x}\Big) \\
& \qquad
-2^{2-2s}(1+x)^{2s-2}\Big(1-\frac{2x}{1+x}\Big)\int_0^1\frac{(1-\frac{2x}{1+x}\tau)^{2s-3}}{\tau^s}\,(1-\tau)^{1-s}\;d\tau \\ 
&=
\frac{2^{1-2s}\Gamma(1-s)^2}{\Gamma(2-2s)}(1+x)^{2s-2}\hf\Big(2-2s,1-s;2-2s\Big|\frac{2x}{1+x}\Big) \\
& \qquad
-\frac{2^{2-2s}\Gamma(1-s)\Gamma(2-s)}{\Gamma(3-2s)}(1+x)^{2s-3}(1-x)\hf\Big(3-2s,1-s;3-2s\Big|\frac{2x}{1+x}\Big) \\ 
&=
\frac{2^{1-2s}\Gamma(1-s)^2}{\Gamma(2-2s)}(1+x)^{2s-2}\Big(1-\frac{2x}{1+x}\Big)^{s-1} \\
& \qquad
-\frac{2^{1-2s}\Gamma(1-s)^2}{\Gamma(2-2s)}(1+x)^{2s-3}(1-x)\Big(1-\frac{2x}{1+x}\Big)^{s-1} \\
&=
\frac{2^{1-2s}\Gamma(1-s)^2}{\Gamma(2-2s)}\bigg[\big(1-x^2\big)^{s-1}-(1+x)^{s-2}(1-x)^s\bigg] \\
&=
\frac{2^{1-2s}\Gamma(1-s)^2}{\Gamma(2-2s)}\big(1-x^2\big)^{s-2}(1-x)\cdot 2x.
\end{align*}
We then deduce
\begin{align*}
\int_1^\infty\frac{(z-x)^{2s-3}+(z+x)^{2s-3}}{z\,\big(z^2-1\big)^s}\;dz \geq
\frac{2^{1-2s}\Gamma(1-s)^2}{\Gamma(2-2s)}\big(1-x^2\big)^{s-2}\big[1-2(1-x)\big].
\end{align*}
The constant in front the above expression can be transformed using the identities on the Gamma function, namely the Legendre duplication formula
\[
\Gamma(1-s)\Gamma\Big(\frac32-s\Big)=2^{1-2(1-s)}\sqrt\pi\,\Gamma(2-2s),
\]
concluding the proof.
\end{proof}

\section{The one-dimensional case: \texorpdfstring{$s=3/4$}{s=3/4}}
\label{logarithmic case}

Note first, that in this case it follows from \eqref{Ftau} and \eqref{defi:j-tau} that
\[
J_{2s-1}(x;y)\Big|_{s=\frac34,n=1}=J_{\frac{1}{2}}(x;y)=-\frac{1}{2\pi}\ln|x^2-y^2|\qquad \text{for }  x,y\in\R,\ x\neq y.
\]
Moreover (\textit{cf}. \eqref{new bound on lambda n=1} and \eqref{poisson})
\begin{align*}
\Lambda\Big(1,\frac34\Big)
&=\frac{2\Gamma(\frac74)}{\Gamma(\frac94)},\\ 
P_{\frac34}(x,z)&=\frac{\sqrt{2}}{2\pi}\frac{\big(1-x^2\big)^{3/4}}{|x-z|\big(z^2-1\big)^{3/4}},\quad\text{for $x\in(-1,1)$, $|z|>1$,}\\
J_{s-1}(x;y)\Big|_{s=\frac34,n=1}&=J_{-\frac{1}{4}}(x;z)=
\frac{1}{4\sqrt{2\pi}}\Big(|x-y|^{-3/2}+|x+y|^{-3/2}\Big)\quad \text{for }x,y\in\R,\ x\neq z.
\end{align*}
As our goal is to verify \eqref{Qbis}, we have to prove
\begin{align*}
-2\ln x+\ln(1-x^2)+\frac{1}{2\Lambda(1,\frac34)\sqrt{\pi}}
\int_1^\infty\frac{(z-x)^{-3/2}+(z+x)^{-3/2}}{z\big(z^2-1\big)^{3/4}}\;dz\geq 0
\qquad\text{for }x_*\Big(1,\frac34\Big)=\frac23<x<1.
\end{align*}
To this aim, we estimate the integral above using Lemma \ref{lem:estimate integral},
which gives us
\[
\int_1^\infty\frac{(z-x)^{-3/2}+(z+x)^{-3/2}}{z\big(z^2-1\big)^{3/4}}\;dz
\geq 
\frac{\Gamma(\frac14)\sqrt\pi}{\Gamma(\frac34)}\big(1-x^2\big)^{-5/4}\big(2x-1\big).
\]
In this way, we are left with verifying
\begin{multline*}
-2\ln x+\ln(1-x^2)
+\frac{\Gamma(\frac94)}{4\Gamma(\frac74)}\frac{\Gamma(\frac14)}{\Gamma(\frac34)}\big(1-x^2\big)^{-5/4}\big(2x-1\big)=\\
=-2\ln x+\ln(1-x^2)
+\frac{5\,\Gamma(\frac14)^2}{48\,\Gamma(\frac34)^2}\big(1-x^2\big)^{-5/4}\big(2x-1\big)
\geq 0
\qquad\text{for }\frac23<x<1.
\end{multline*}
Using the fact that
\[
2x-2\geq-\frac32(1-x^2)
\quad\text{and}\quad\big(1-x^2\big)^{-5/4}\geq \big(1-x^2\big)^{-1}
\qquad\text{for }\frac23<x<1
\]
in the following we rather show
\begin{align*}
-2\ln x+\ln(1-x^2)+c\big(1-x^2\big)^{-1}\Big(1-\frac32(1-x^2)\Big)
\geq 0
\qquad\text{for }\frac23<x<1,\
c = \frac{5\,\Gamma(\frac14)^2}{48\,\Gamma(\frac34)^2}=0.91...
\end{align*}
We do so by re-labeling $t=1-x^2$ and by checking 
\begin{align}\label{54545454545}
-\ln(1-t)+\ln t+ct^{-1}\Big(1-\frac32t\Big)\geq 0
\qquad\text{for }0<t<\frac59.
\end{align}
The above inequality follows by computing the minimum of the left-hand side in the given range for $t$. Indeed,
\begin{align*}
\frac{d}{dt}\bigg[-\ln(1-t)+\ln t+ct^{-1}\Big(1-\frac32t\Big)\bigg]=\frac1{1-t}+\frac1t-\frac{c}{t^2}\geq 0 \qquad\text{for }0<t<\frac59
\end{align*}
if and only 
\begin{align*}
0.4769...=\frac{c}{c+1}\leq t<\frac59
\end{align*}
so that the left-hand side of \eqref{54545454545} attains its minimum at $c/(c+1)$ where it equals
\begin{align*}
-\ln(1-\frac{c}{c+1})+\ln \frac{c}{c+1}+(c+1)\Big(1-\frac32\frac{c}{c+1}\Big)=\ln c+1-\frac{c}2=0.45...>0.
\end{align*}
Hence, \eqref{Qbis} holds for $n=1$ and $s=\frac34$.

\section{The general one-dimensional case}\label{1 s}

In the following, we analyze \eqref{Qbis} with $s\neq \frac{3}{4}$ and $n=1$. Recall the definition of $x_*(1,s)$ in Lemma \ref{lem:midconcavity}. 
As an application of Lemma \ref{lem:estimate integral} and of definitions \eqref{Ftau} and \eqref{poisson}
we rather check that 
\begin{multline*}
\frac{2x^{4s-3}-(1-x)^{4s-3}-(1+x)^{4s-3}}{3-4s}+\mu(1-x^2)^{s-2}\big(2x-1\big)\geq 0, \\
\text{for }s\in\Big(\frac12,1\Big)\setminus\Big\{\frac34\Big\},\ x_*(1,s)<x<1,
\end{multline*}
where (recall \eqref{bigger lambda} and \eqref{new bound on lambda n=1})
\begin{align*}
\mu=\frac1{2^{2s-1}\Lambda\Gamma(s)|\Gamma(s-1)|}\frac{2^{4s-2}\sqrt\pi\,\Gamma(2s-1)}{\Gamma(\frac52-2s)}=\frac{2^{2s-1}(1-s)\sqrt\pi}{s}
\frac{\Gamma(\frac32+s)\Gamma(2s-1)}{\Gamma(s)^3\,\Gamma(\frac32+2s)\Gamma(\frac52-2s)}.
\end{align*}

Note that the function 
\begin{align*}
(0,1]\ni x\longmapsto \frac{1}{3-4s}\Big(2{x}^{4s-3}-{(1+x)}^{4s-3}\Big)
\end{align*}
is decreasing. Indeed, this follows by differentiation:
\begin{align*}
\frac{1}{3-4s}\frac{d}{dx}\Big(2{x}^{4s-3}-{(1+x)}^{4s-3}\Big)= -\Big(2{x}^{4(s-1)}-{(1+x)}^{4(s-1)}\Big)<0
\qquad\text{for }x\in(0,1].
\end{align*}
Then, fixing $a,b\in[\frac35,1]$ with $a<b$ we find with this for $x\in(a,b)$
\begin{align*}
&\frac{2x^{4s-3}-(1-x)^{4s-3}-(1+x)^{4s-3}}{3-4s}+\mu(1-x^2)^{s-2}(2x-1)\\
&\qquad\geq \frac{2b^{4s-3}-(1+b)^{4s-3}}{3-4s}-\frac{(1-x)^{4s-3}}{3-4s}+(1+b)^{s-2}(2a-1)\mu(1-x)^{s-2}=:q_{a,b}(s,x).
\end{align*}
A direct computation gives that the function $(a,b)\ni x\mapsto q_{a,b}(s,x)$
is controlled from below in $(a,b)$ by the value
\[
q_{a,b}\big(s,x_{a,b}(s)\big),
\qquad \text{where }
x_{a,b}(s):=1-\Big((1+b)^{s-2}(2a-1)(2-s)\mu\Big)^{\frac{1}{3s-1}}.
\]
Keeping this in mind, we split
\begin{gather*}
\Big(\frac35,1\Big) =
\big(a_1,b_1\big]\cup\big(a_2,b_2\big]\cup\big(a_3,b_3\big]\cup\big(a_4,b_4\big) \\
a_1=\frac35,\ b_1=a_2=\frac7{10},\ b_2=a_3=\frac45,\ b_3=a_4=\frac9{10},\ b_4=1.
\end{gather*}
In each of these subintervals it holds that 
\[
q_{a_i,b_i}(s,x)\geq q_{a_i,b_i}\big(s,x_{a_i,b_i}(s)\big)>0
\qquad \text{for }x\in(a_i,b_i),\ s\in\Big(\frac12,1\Big)\setminus\Big\{\frac34\Big\},\ i\in\{1,2,3,4\},
\]
see Figure \ref{plot-of-ms1}. From this it follows that \eqref{Qbis} holds for $n=1$.
\begin{figure}
\centering
\begin{minipage}{.49\textwidth}
\includegraphics[width=0.9\columnwidth]{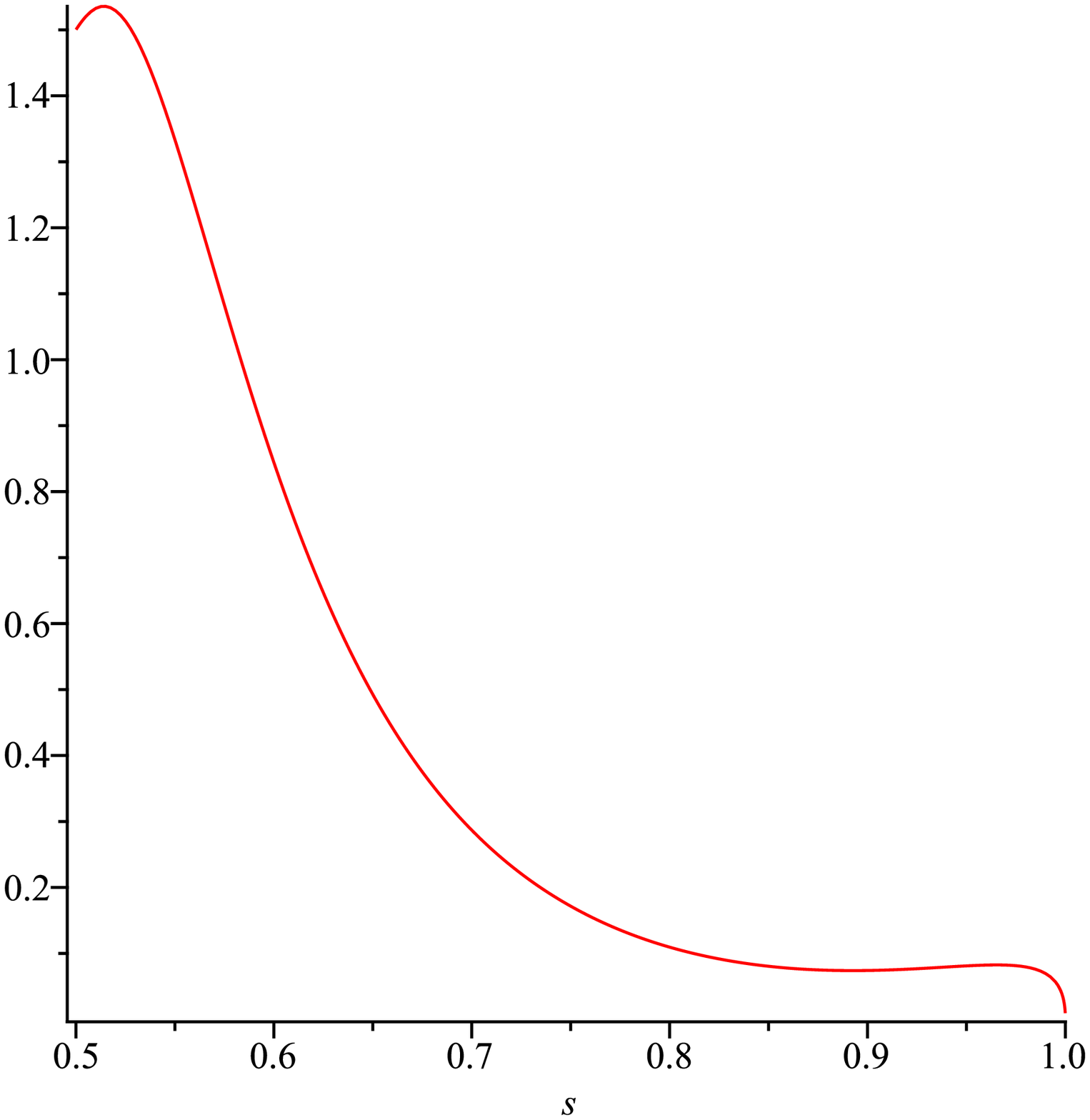}\\
\includegraphics[width=0.9\columnwidth]{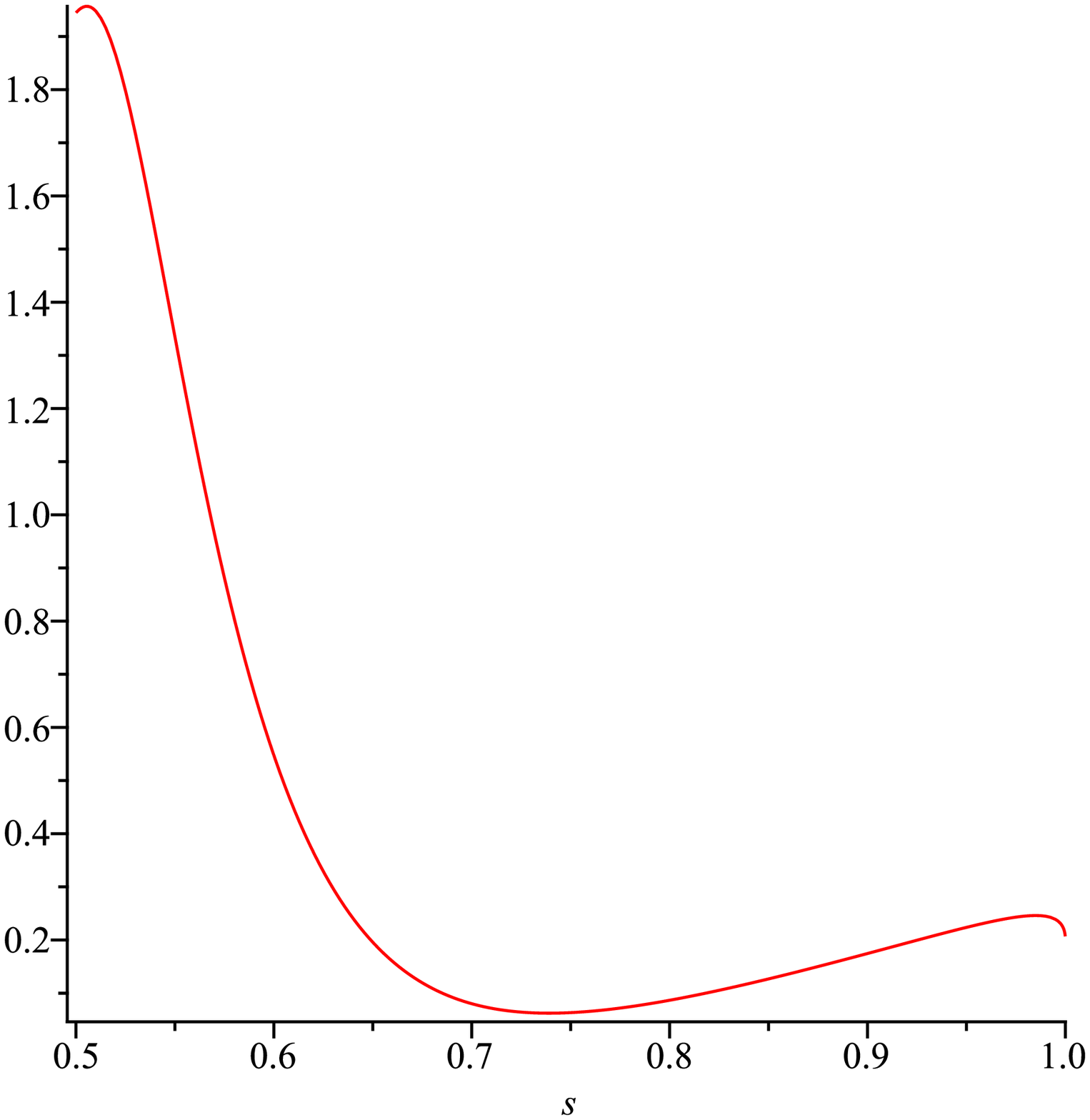}
\end{minipage}
\begin{minipage}{.49\textwidth}
\includegraphics[width=0.9\columnwidth]{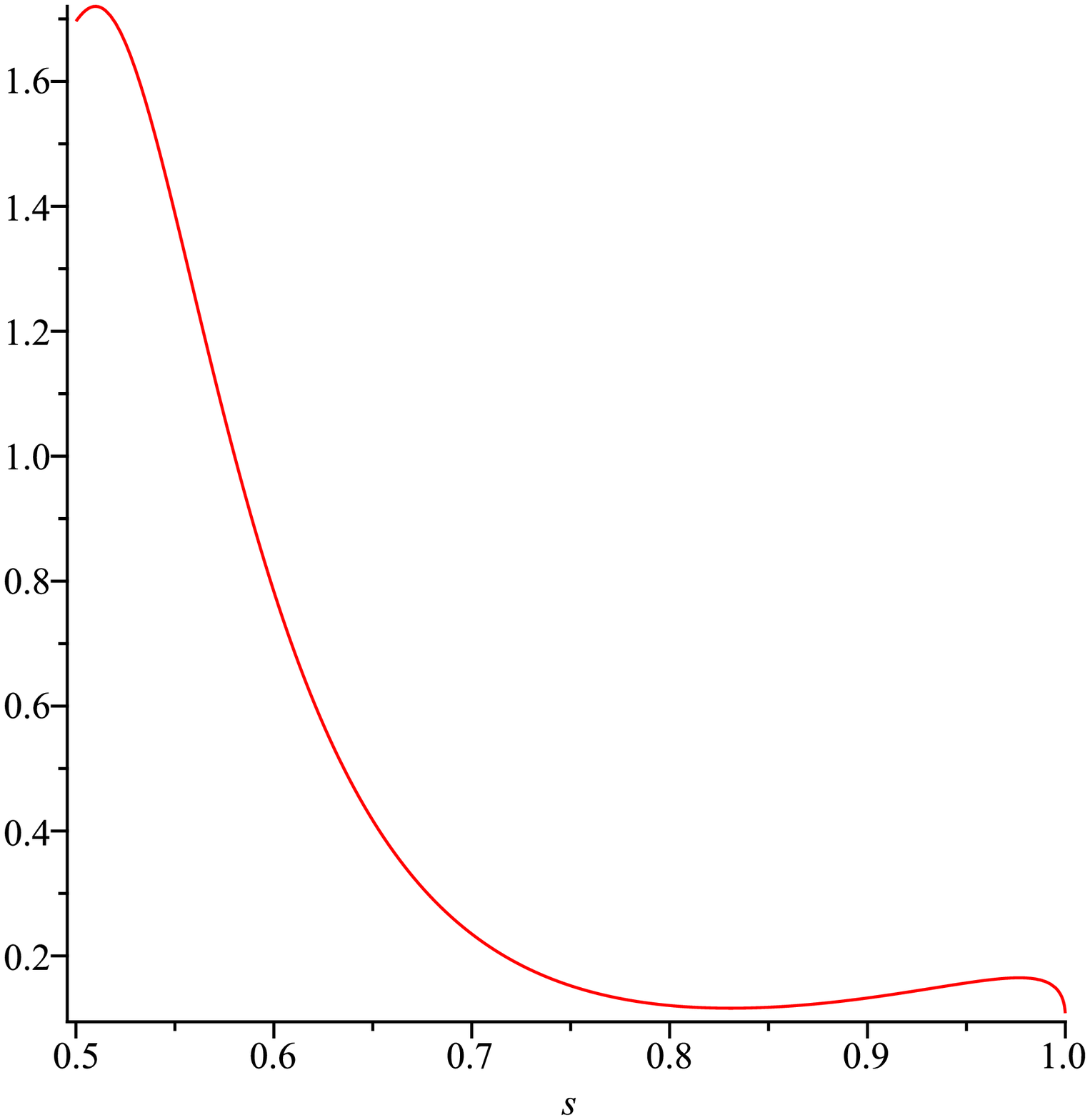}\\
\includegraphics[width=0.9\columnwidth]{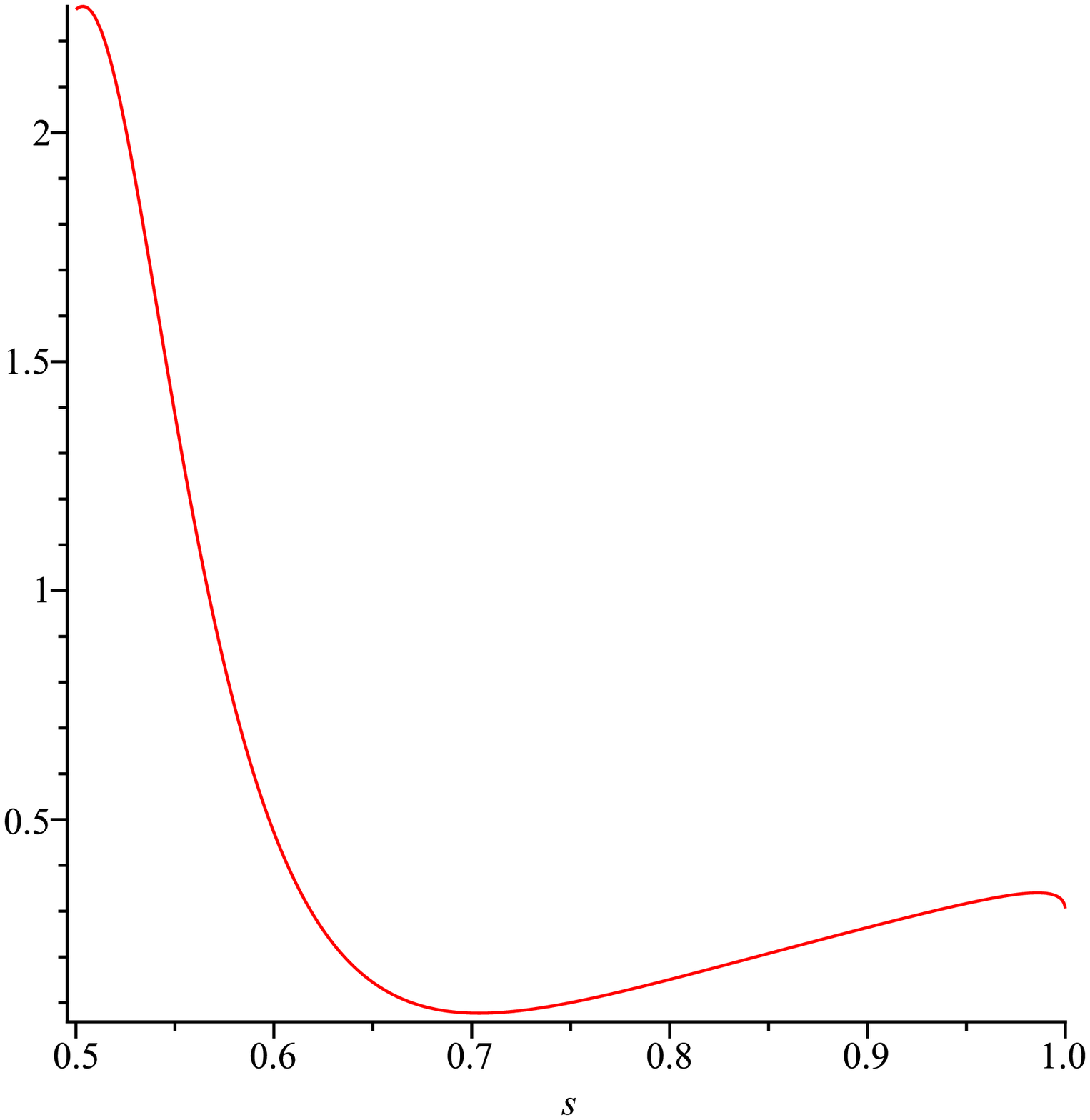}
\end{minipage}
\caption{A graph of $q_{a,b}(s,x_{a,b}(s))$ for $\frac12<s<1$ with $(a,b)=(0.9,1)$---top left---, $(a,b)=(0.8,0.9)$---top right---, $(a,b)=(0.7,8)$---bottom left---, and $(a,b)=(0.6,0.7)$---bottom right.}\label{plot-of-ms1}
\end{figure}

\section{The higher-dimensional case}\label{n s}

In the following we test the validity of Theorem \ref{main theorem} for $2\leq 12$. In view of Lemma \ref{lem:midconcavity}, it remains to show that, for $x\in(x_*(n,s),1)$ and $x_*$ as in Lemma \ref{lem:midconcavity}, it holds
\begin{multline}\label{q-bis-higherA}
\int_{-1}^1(1-t^2)^{\frac{n-3}{2}}\Bigg(x^{4s-2-n}-\frac{1}{(x^2+1-2xt)^{\frac{n}{2}+1-2s}}\Bigg)\;dt\\
+M(n,s)\int_{1}^{\infty}\frac{1}{\big(r^2-1\big)^s\,r}\int_{-1}^1  \frac{(1-t^2)^{\frac{n-3}{2}}}{(x^2+r^2-2xrt)^{\frac{n}{2}-s+1}}\;dt\;dr\geq 0
\end{multline}
with (recall \eqref{poisson}, \eqref{Ftau}, and \eqref{bigger lambda})
\begin{align*}
M(n,s) &:=
\frac{2\pi^{\frac{n}{2}}}{\Gamma(\frac{n}{2})} \frac{\gamma(n,s)\,\kappa(n,s-1)}{\Lambda(n,s)\,\kappa(n,2s-1)} \\
&=
\frac{2}{\Gamma(s)\,\Gamma(1-s)} 
\frac{(\frac{n}{2}+s)\,\Gamma(\frac{n}2)\,\Gamma(1+2s)}{2^{2s}\Gamma(1+s)^2\,\Gamma(\frac{n}{2}+1+2s)} 
\frac{\Gamma(\frac{n}2+1-s)}{2^{2s-3}\big|\Gamma(s-1)\big|}
\frac{2^{4s-3}\Gamma(2s-1)}{\Gamma(\frac{n}2+1-2s)} \\
&=\frac{(1-s)\Gamma(1+2s)\Gamma(2s-1)}{s^2\Gamma(s)^4\Gamma(1-s)}\cdot\frac{(n+2s)\Gamma(\frac{n}{2}) \Gamma(\frac{n}{2}+1-s)}{\Gamma(\frac{n}{2}+1+2s)\Gamma(\frac{n}{2}+1-2s)}
\end{align*}
where we have used Lemma \ref{lem:integral computation}, the transformation into polar coordinates, and some reformulations of the constant using properties of the Gamma function.
Note here, that
\[
\int_{-1}^1(1-t^2)^{\frac{n-3}{2}}\ dt=\int_{0}^1\tau^{-\frac{1}{2}}(1-\tau)^{\frac{n-3}{2}}\ d\tau=\frac{\sqrt{\pi}\Gamma(\frac{n-1}{2})}{\Gamma(\frac{n}{2})}=\frac{2^{n-2}\Gamma(\frac{n-1}{2})^2}{\Gamma(n-1)}
\]
and with the transformation $t+1=2\tau$ we have
\begin{align}
\int_{-1}^1\frac{(1-t^2)^{\frac{n-3}{2}}}{\big(x^2+1-2xt\big)^{\frac{n}{2}+1-2s}}\;dt
&=2\int_{0}^1\frac{(2-2\tau)^{\frac{n-3}{2}}(2\tau)^{\frac{n-3}{2}}}{\big(x^2+1-2x(2\tau-1)\big)^{\frac{n}{2}+1-2s}}\;dt\notag\\
&=2^{n-2}\int_{0}^1\frac{(1-\tau)^{\frac{n-3}{2}}\tau^{\frac{n-3}{2}}}{\big((x+1)^2-4x\tau\big)^{\frac{n}{2}+1-2s}}\;dt\notag\\
&=\frac{2^{n-2}}{(x+1)^{n+2-4s}}\int_{0}^1\tau^{\frac{n-3}{2}}(1-\tau)^{\frac{n-3}{2}}\Big(1-\frac{4x}{(x+1)^2}\tau\Big)^{2s-\frac{n}{2}-1}\;dt\notag\\
&=\frac{2^{n-2}}{(x+1)^{n+2-4s}}\frac{\Gamma(\frac{n-1}{2})^2}{\Gamma(n-1)}\hf\bigg(\frac{n}{2}+1-2s,\frac{n-1}{2};n-1\Big|\frac{4x}{(x+1)^2}\bigg),\label{integral A}
\end{align}
where we have used \eqref{hyper def}. Similarly, we also have 
\begin{align}
\int_{-1}^1  \frac{(1-t^2)^{\frac{n-3}{2}}}{\big(x^2+r^2+2xrt\big)^{\frac{n}{2}-s+1}}\;\ dt&=\int_{-1}^1  \frac{(1-t^2)^{\frac{n-3}{2}}}{\big(x^2+r^2-2xrt\big)^{\frac{n}{2}-s+1}}\;dt\notag\\
&=\frac{2^{n-2}}{(x+r)^{n+2-2s}}\int_{0}^1\tau^{\frac{n-3}{2}}(1-\tau)^{\frac{n-3}{2}}\Big(1-\frac{4xr}{(x+r)^2}\tau\Big)^{ s-\frac{n}{2}-1}\;dt\notag\\
&=\frac{2^{n-2}}{(x+r)^{n+2-2s}}\frac{\Gamma(\frac{n-1}{2})^2}{\Gamma(n-1)}\hf\bigg(\frac{n}{2}+1-s,\frac{n-1}{2};n-1\Big|\frac{4xr}{(x+r)^2}\bigg).\label{integral B}
\end{align}
We now perform some transformations on the hypergeometric functions appearing respectively in \eqref{integral A} and \eqref{integral B}. For this, let $\sigma=1-2s$ or $\sigma=1-s$ and let $t=1$ or $t=r$ (so that $t\geq 1$). Then note that
\[
\frac{n}{2}+\sigma-\frac{n-1}{2}=\frac{n}{2}+\sigma +\frac{n-1}{2}-(n-1)
\qquad\text{and}\qquad
\sqrt{1-\frac{4xt}{(x+t)^2}}=\frac{t-x}{t+x},
\]
so that, by \eqref{hyper identity 3}, we have
\begin{align*}
&\hf\bigg(\frac{n}{2}+\sigma,\frac{n-1}{2};n-1\Big|\frac{4xt}{(x+t)^2}\bigg)=\\
&=2^{2\sigma+n}\Big(1+\frac{t-x}{t+x}\Big)^{-2\sigma-n}\hf\bigg(\frac{n}{2}+\sigma,\frac{n}{2}+\sigma-\frac{n-1}{2}+\frac12;\frac{n-1}{2}+\frac12\Big|\Big(\frac{1-\frac{t-x}{t+x}}{1+\frac{t-x}{t+x}}\Big)^2\bigg)\\
&=\Big(\frac{t}{t+x}\Big)^{-2\sigma-n}\hf\bigg(\frac{n}{2}+\sigma,1+\sigma;\frac{n}{2}\Big| \frac{x^2}{t^2}\bigg).
\end{align*}
With this, \eqref{q-bis-higherA} translates to 
\begin{multline}\label{q-bis-higherB}
x^{4s-2-n}-\hf\bigg(\frac{n}{2}+1-2s,2-2s;\frac{n}{2}\Big|x^2\bigg) \\
+M(n,s)\int_1^\infty\frac{r^{2s-n-3}}{(r^2-1)^s } \hf\Big(\frac{n}{2}+1-s,2-s;\frac{n}{2}\Big| \frac{x^2}{r^2}\Big)\;dr\geq 0.
\end{multline}
We exploit next the series expansion of the hypergeometric function, see \eqref{series def}.
We have, due to the absolute convergence of the integral and the involved infinite sum,
\begin{align*}
& \int_1^\infty\frac{r^{2s-n-3}}{(r^2-1)^s}\hf\Big(\frac{n}{2}+1-s,2-s;\frac{n}{2}\Big| \frac{x^2}{r^2}\Big)\;dr=\\
& =\frac{\Gamma\big(\frac{n}2\big)}{\Gamma\big(\frac{n}{2}+1-s\big)\,\Gamma(2-s)}\sum_{k=0}^\infty
\frac{\Gamma\big(\frac{n}{2}+1-s+k\big)\,\Gamma(2-s+k)}{\Gamma\big(\frac{n}2+k\big)}\frac{x^{2k}}{k!}
\int_1^\infty\frac{r^{2s-n-3-2k}}{(r^2-1)^s}\;dr \\
& =\frac{\Gamma\big(\frac{n}2\big)}{2\Gamma\big(\frac{n}{2}+1-s\big)\,\Gamma(2-s)}\sum_{k=0}^\infty
\frac{\Gamma\big(\frac{n}{2}+1-s+k\big)\,\Gamma(2-s+k)}{\Gamma\big(\frac{n}2+k\big)}\frac{x^{2k}}{k!}
\int_0^1\frac{\rho^{n/2+k}}{(1-\rho)^s}\;d\rho \\
& =\frac{\Gamma\big(\frac{n}2\big)}{2\Gamma\big(\frac{n}{2}+1-s\big)\,\Gamma(2-s)}\sum_{k=0}^\infty
\frac{\Gamma\big(\frac{n}{2}+1-s+k\big)\,\Gamma(2-s+k)}{\Gamma\big(\frac{n}2+k\big)}\frac{x^{2k}}{k!}
\frac{\Gamma\big(\frac{n}2+k+1\big)\,\Gamma(1-s)}{\Gamma\big(\frac{n}2+k+2-s\big)} \\
& =\frac{\Gamma\big(\frac{n}2\big)}{2\Gamma\big(\frac{n}{2}+1-s\big)\,(1-s)}\sum_{k=0}^\infty
\frac{\big(\frac{n}2+k\big)\Gamma(2-s+k)}{\frac{n}2+k+1-s}\frac{x^{2k}}{k!} 
\end{align*}
where we have used a change of variables with $\rho=r^{-2}$ and \eqref{beta-function}.
As it holds
\begin{align}\label{898989898}
\frac{n}{n+2-2s}\leq\frac{\frac{n}2+k}{\frac{n}2+k+1-s}\leq 1
\qquad \text{for any }n\in\N,\ k\in\N\cup\{0\},\ s\in\Big(\frac12,1\Big),
\end{align}
we deduce from the above calculation that
\[
\frac{\Gamma\big(\frac{n}2+1\big)\,\Gamma(1-s)}{2\Gamma\big(\frac{n}{2}+2-s\big)}\big(1-x^2\big)^{s-2}
\leq
\int_1^\infty\frac{r^{2s-n-3}}{(r^2-1)^s}\hf\Big(\frac{n}{2}+1-s,2-s;\frac{n}{2}\Big| \frac{x^2}{r^2}\Big)\;dr
\leq
\frac{\Gamma\big(\frac{n}2\big)\,\Gamma(1-s)}{2\Gamma\big(\frac{n}{2}+1-s\big)}\big(1-x^2\big)^{s-2}
\]
Indeed, it holds
\[
\sum_{k=0}^\infty\frac{\Gamma(2-s+k)}{\Gamma(2-s)}\frac{x^{2k}}{k!}
=\hf\big(2-s,1;1\big|x^2\big)
=\big(1-x^2\big)^{s-2},
\]
see \eqref{hyper identity 1}. Hence, \eqref{q-bis-higherB} is satisfied once we have
\begin{align}\label{q-bis-higherC}
x^{4s-2-n}-\hf\bigg(\frac{n}{2}+1-2s,2-2s;\frac{n}{2}\Big|x^2\bigg) +\frac{M(n,s)\Gamma\big(\frac{n}2+1\big)\,\Gamma(1-s)}{2\Gamma\big(\frac{n}{2}+2-s\big)}\big(1-x^2\big)^{s-2}\geq 0.
\end{align}
In Figure \ref{fig:low-dim} we present the plots of the left-hand side of \eqref{q-bis-higherC} for $n=2,\ldots,11$, where this is indeed positive.

\begin{figure}\label{fig:low-dim}
\centering
\begin{minipage}{.49\textwidth}
\centering
\includegraphics[width=.8\columnwidth]{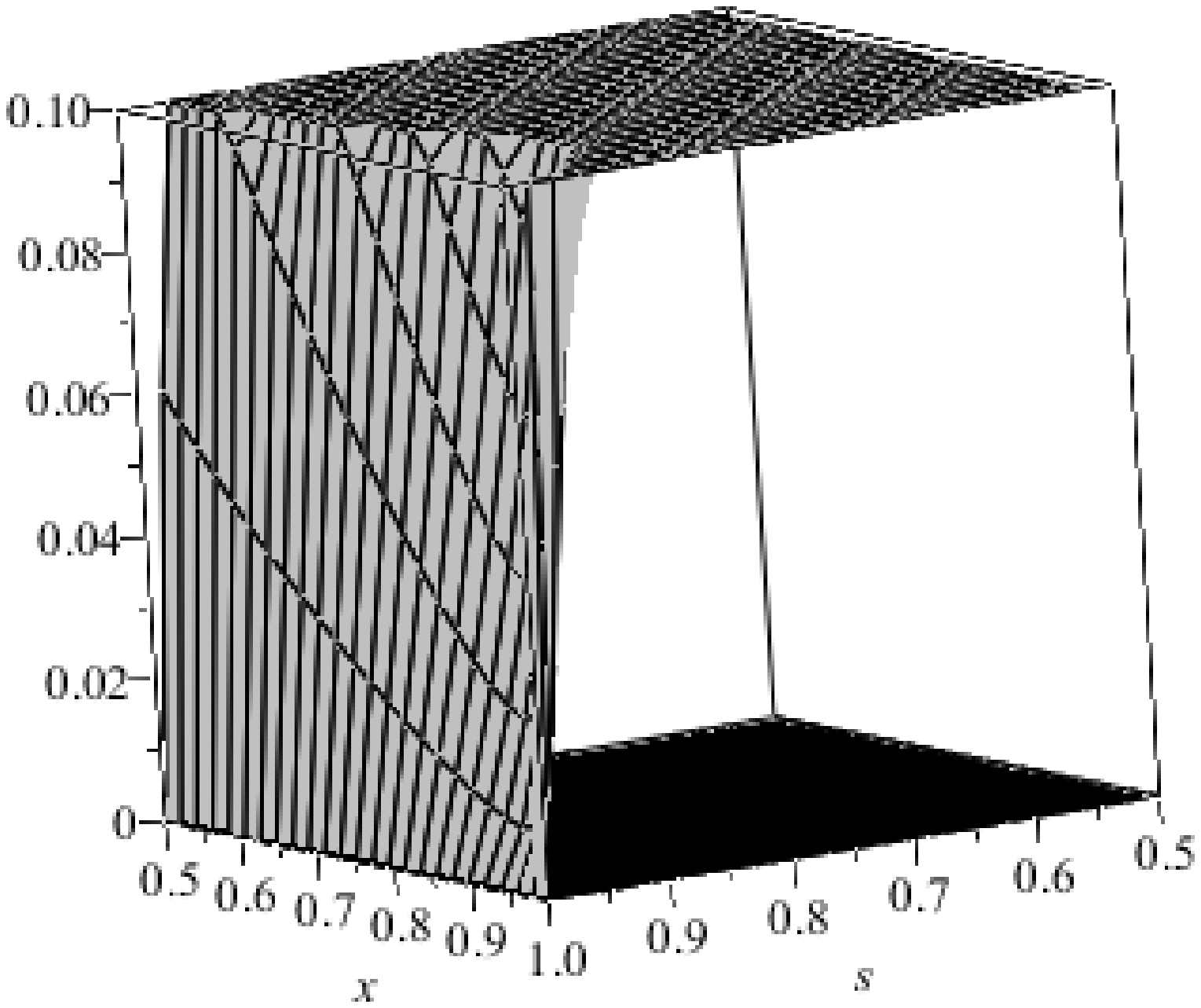}\\
\includegraphics[width=.8\columnwidth]{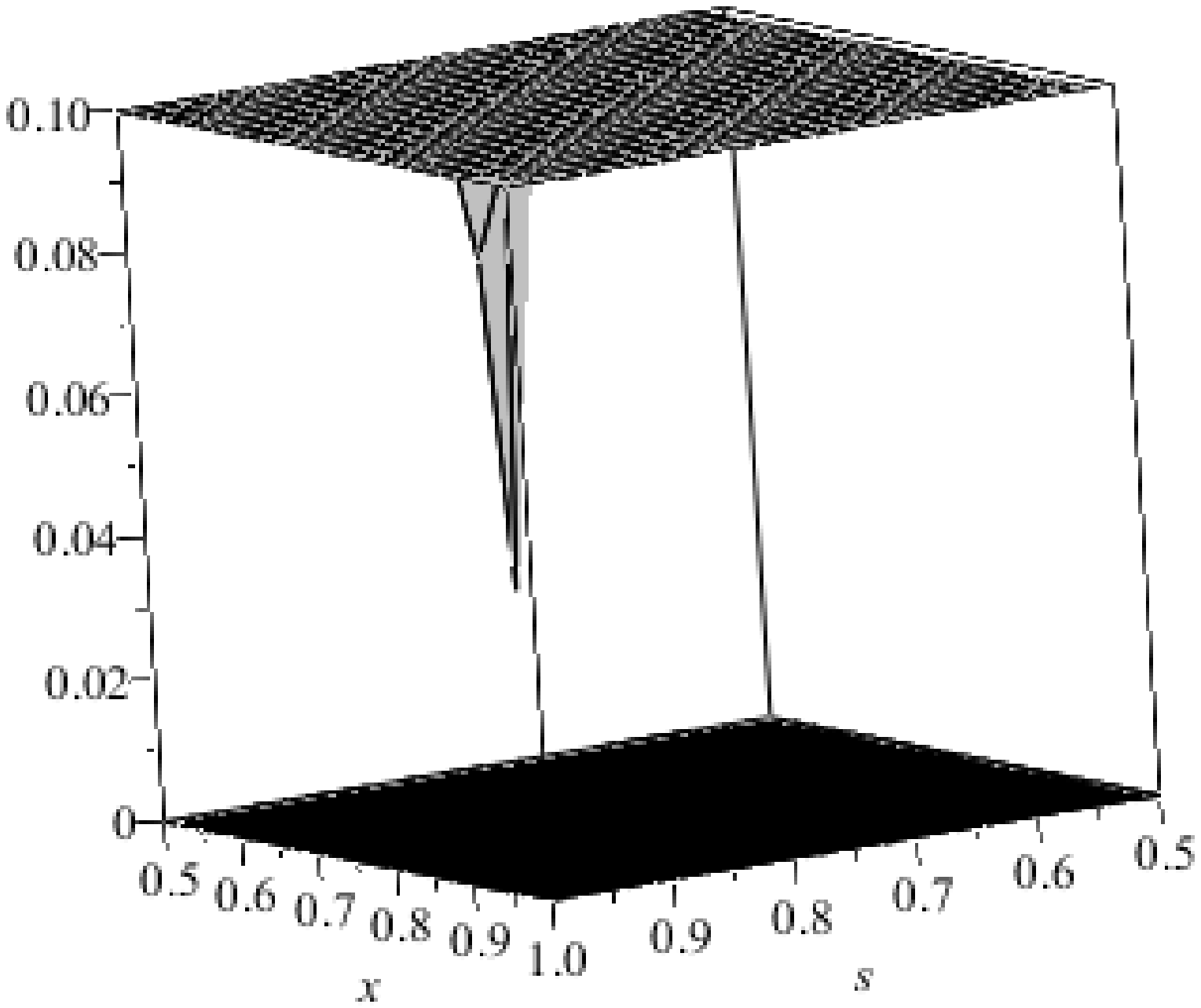}\\
\includegraphics[width=.8\columnwidth]{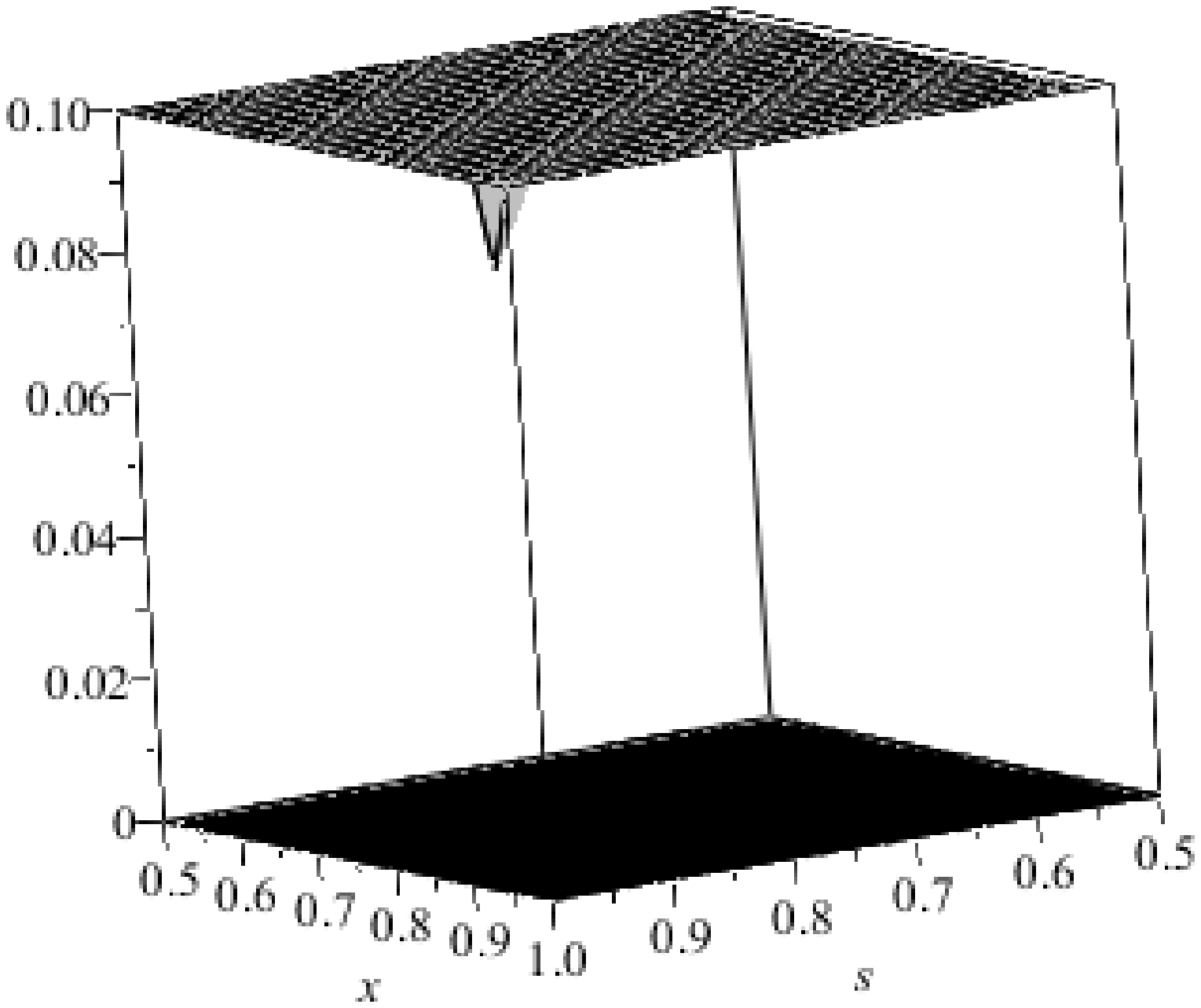}\\
\end{minipage}
\begin{minipage}{.49\textwidth}
\centering
\includegraphics[width=.8\columnwidth]{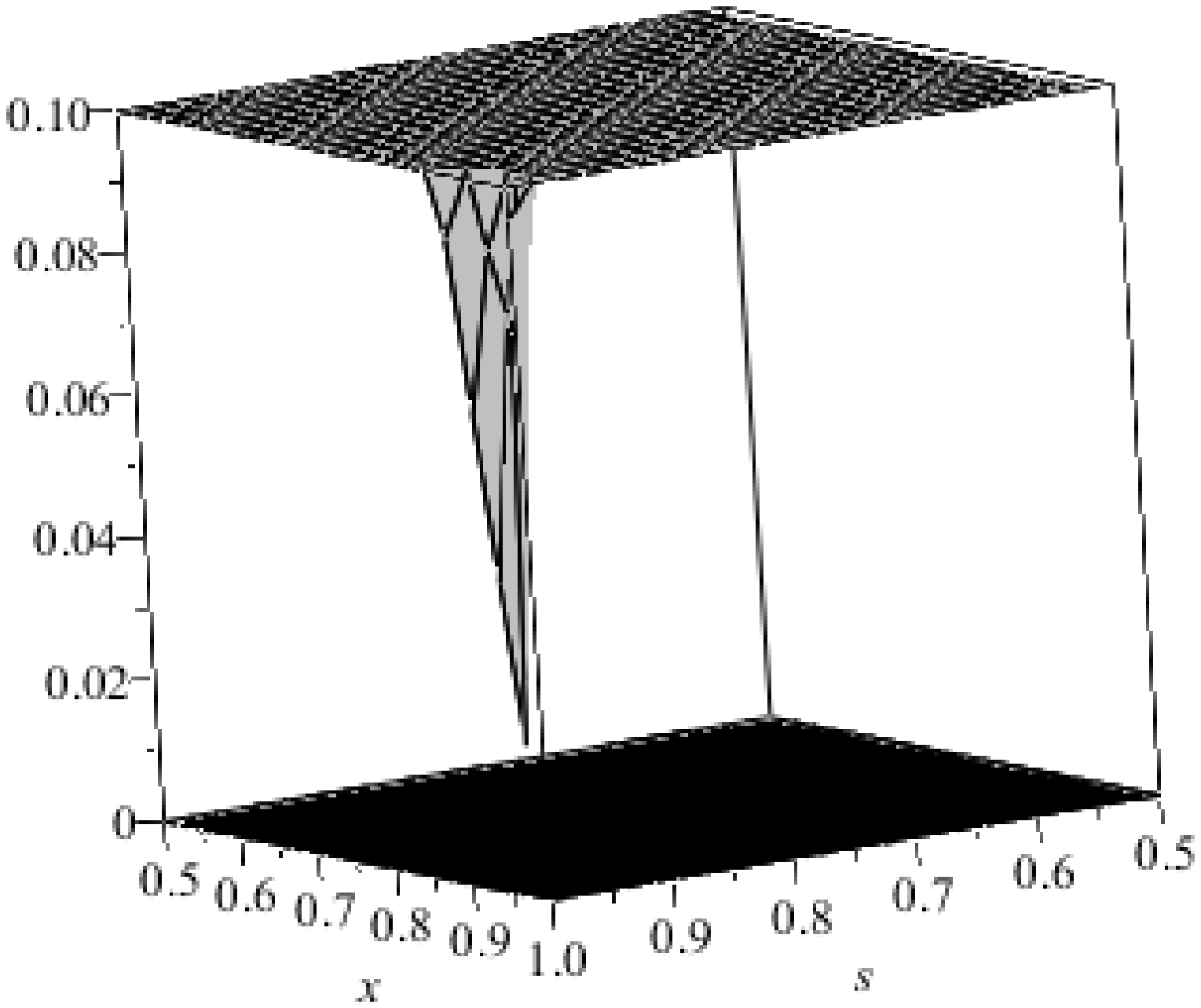}\\
\includegraphics[width=.8\columnwidth]{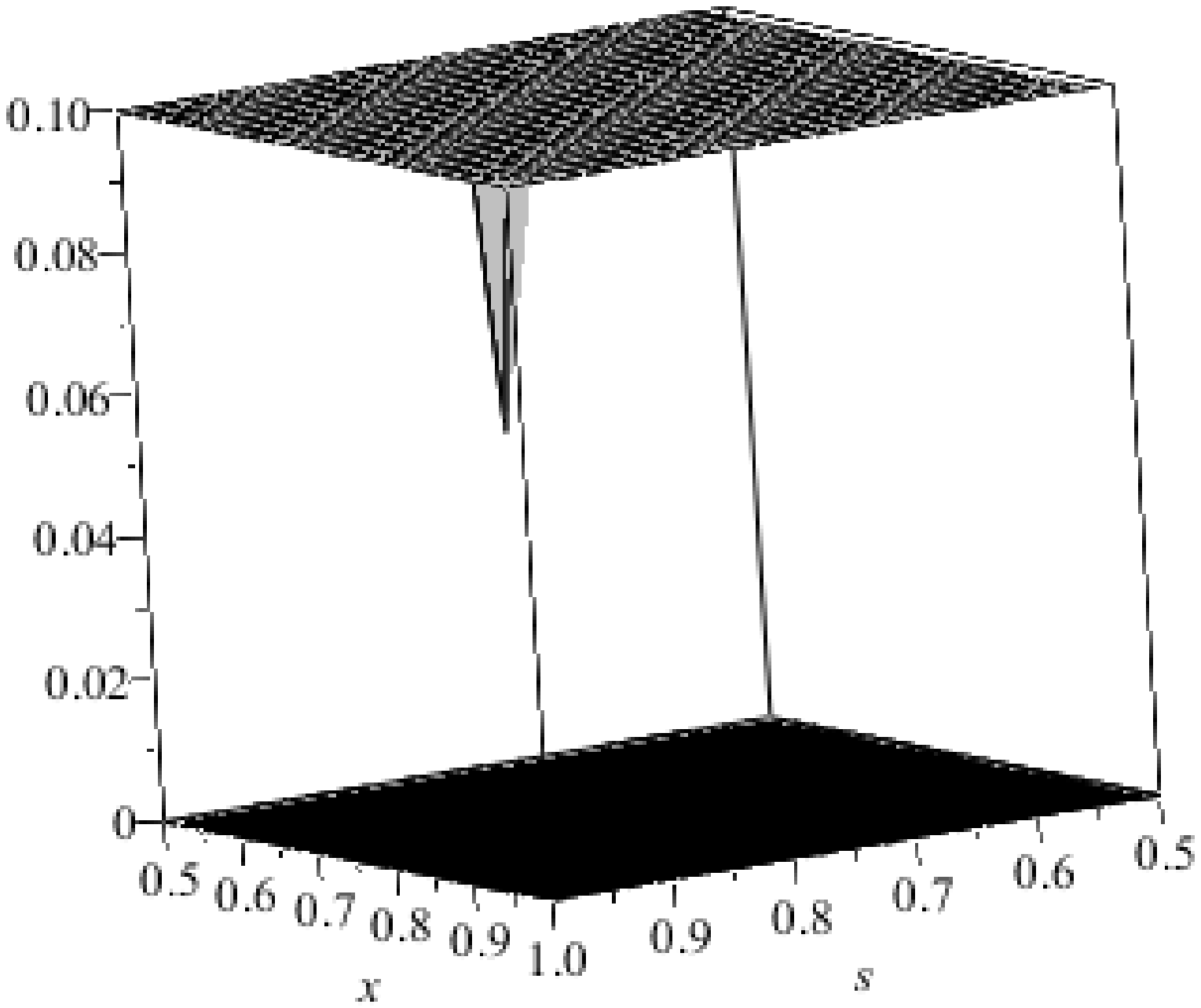}\\
\includegraphics[width=.8\columnwidth]{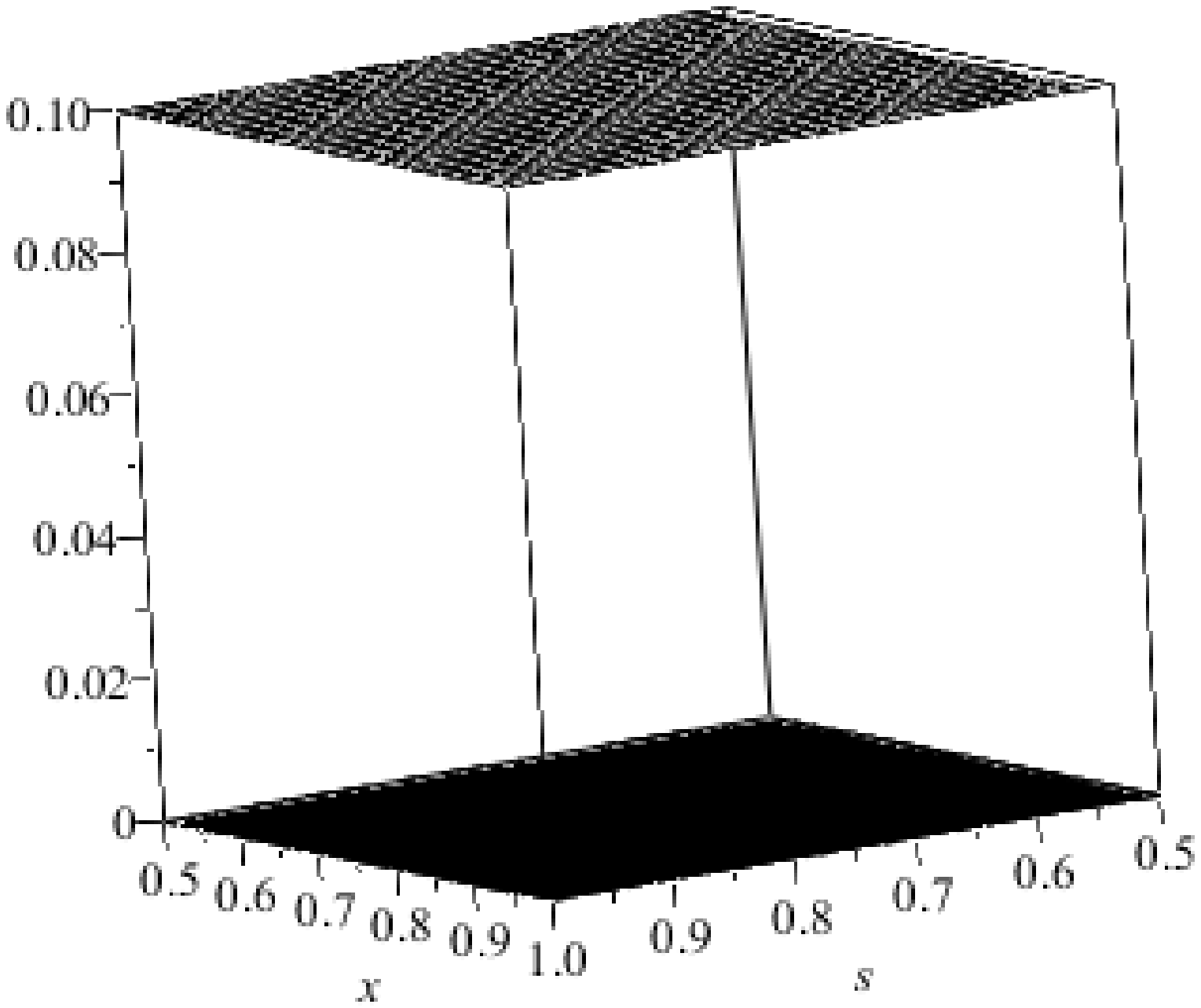}
\end{minipage}
\caption[Certain plots 
for $n=2$ (top left), $n=3$ (top right), $n=4$ (center left), $n=5$ (center right), $n=6$ (bottom left), and $n=7,\ldots,11$ (bottom right).]{The plots of 
\begin{align*}
\Big(\frac12,1\Big)\times\Big(\frac12,1\Big) & \longrightarrow \R \\
(s,x) & \longmapsto \min\Bigg\{\frac1{10},x^{4s-2-n}-\hf\bigg(\frac{n}{2}+1-2s,2-2s;\frac{n}{2}\Big|x^2\bigg) \\
&\qquad\qquad\qquad +\frac{M(n,s)\Gamma\big(\frac{n}2+1\big)\,\Gamma(1-s)}{2\Gamma\big(\frac{n}{2}+2-s\big)}\big(1-x^2\big)^{s-2}\Bigg\}
\end{align*}
for $n=2$ (top left), $n=3$ (top right), $n=4$ (center left), $n=5$ (center right), $n=6$ (bottom left), and $n=7,\ldots,11$ (bottom right).
}
\end{figure}

\begin{remark}
By avoiding estimate \eqref{898989898} and keeping the series expansion of the hypergeometric function, see \eqref{series def}, it is possible to see that also the case $n=12$ is actually covered by this approach. However, for larger $n$ this keeps failing, although there always are some ranges of $s$ where the left-hand side of \eqref{q-bis-higherB} stays positive. Finally, let us mention that for $n\geq 127$, the left-hand side of \eqref{q-bis-higherC} seems to be positive again. Indeed, again with the series expansion of the hypergeometric function, see \eqref{series def}, it holds for $n\geq 4$
\[
\hf\bigg(\frac{n}{2}+1-2s,2-2s;\frac{n}{2}\Big|x^2\bigg) \leq \sum_{k=0}^{\infty}\frac{\Gamma(2-2s+k)}{\Gamma(2-2s)}\frac{x^{2k}}{k!}=\big(1-x^2\big)^{2s-2},
\]
see \eqref{hyper identity 1}. So that it remains to check 
\begin{equation}\label{q-bis-higherD}
x^{4s-2-n}-\big(1-x^2\big)^{2s-2} +\frac{M(n,s)\Gamma\big(\frac{n}2+1\big)\,\Gamma(1-s)}{2\Gamma\big(\frac{n}{2}+2-s\big)}\big(1-x^2\big)^{s-2}\geq 0,
\end{equation}
which seems to stay positive for $n\geq 127$. See Figure \ref{more plots}. 
Note indeed that the left-hand side of \eqref{q-bis-higherD} is greater than
\[
x^{4s-2-n}-\big(1-x^2\big)^{2s-2}
\]
which diverges to $+\infty$ as $n\uparrow\infty$ for $s,x\in(\frac12,1)$.

Let us mention here that our strategy strongly relies also on estimate \eqref{new bound on lambda} and that a more precise estimate here could improve a lot the number of dimensions covered in our analysis.
\end{remark}

\begin{figure}
\centering
\includegraphics[width=.45\columnwidth]{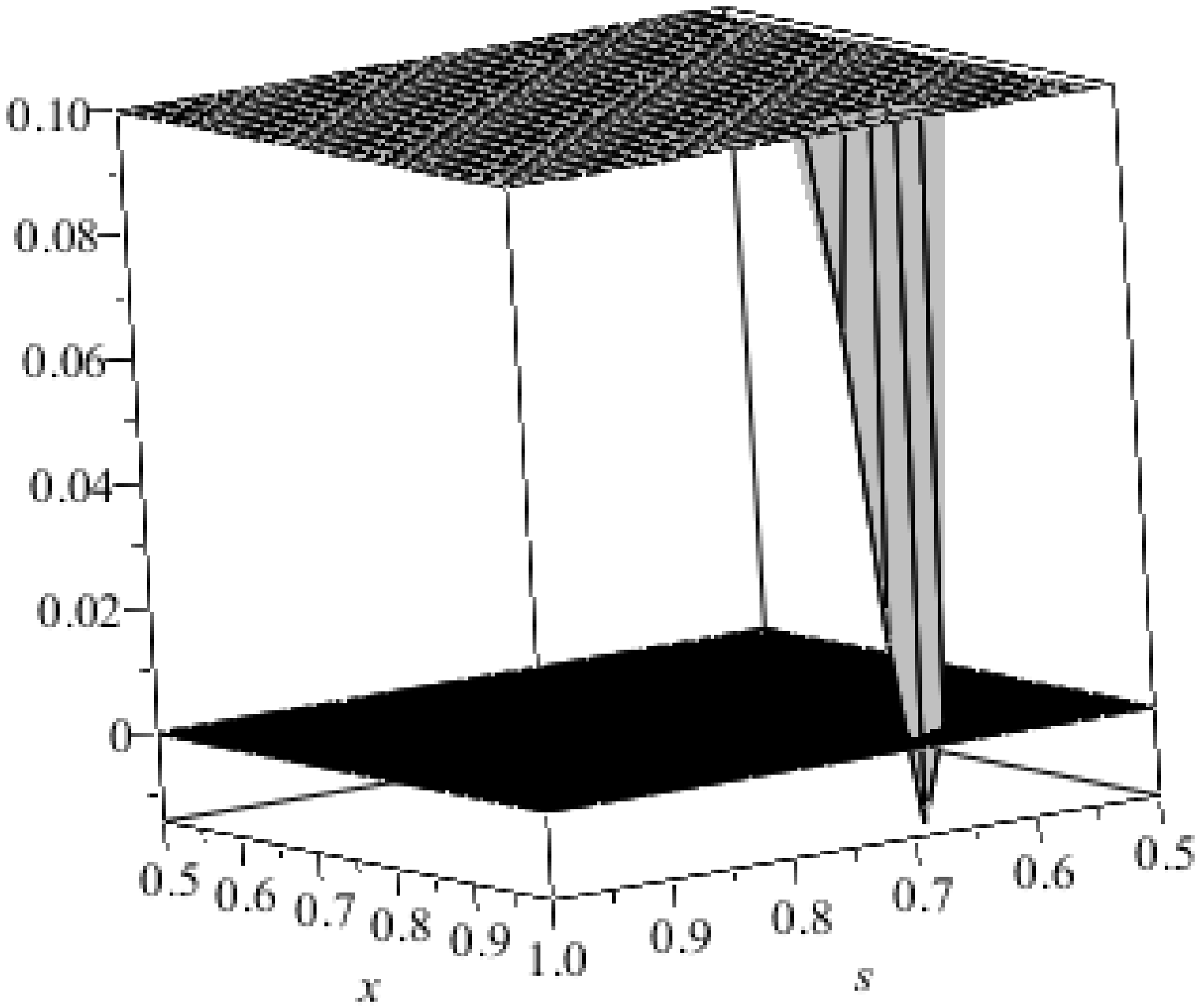}
\includegraphics[width=.45\columnwidth]{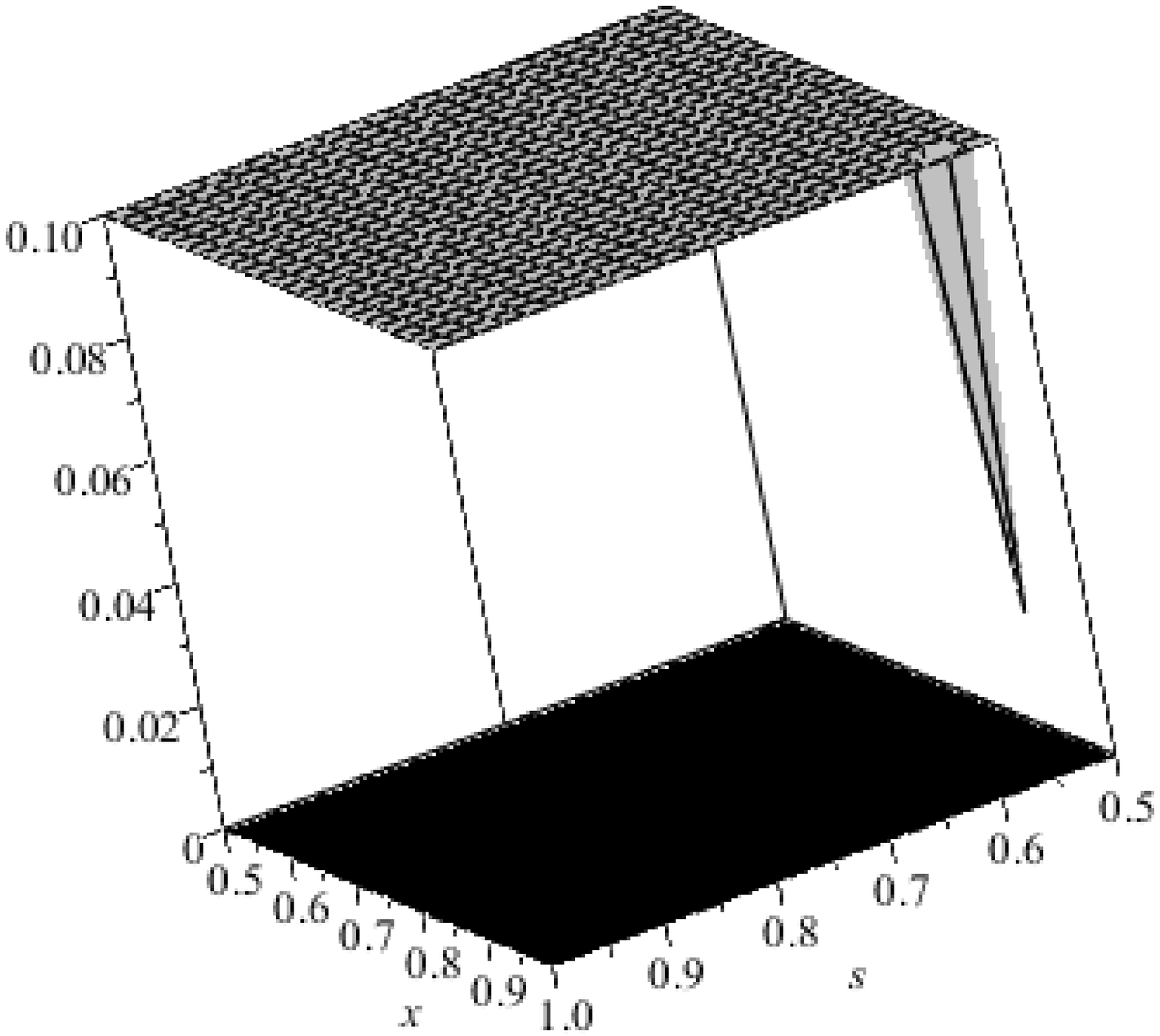}
\caption{On the left, \eqref{q-bis-higherC} fails for $n=12$. On the right, \eqref{q-bis-higherC} is recovered for $n=127$ via the weaker condition \eqref{q-bis-higherD}.}\label{more plots}
\end{figure}

\appendix

\section{Special functions}
\label{special functions}

For the reader's convenience we list here the definitions and some properties about the special functions that we use.

\subsection{The Gamma function} 
As usual, the Gamma function is defined by 
\[
\Gamma(z)=\int_0^{\infty}x^{z-1}e^{-x}\;dx, \qquad\text{for }z>0.
\]
As it satisfies the recursive formula
\begin{align*}
\Gamma(z+1)=z\,\Gamma(z)
\end{align*}
its definition can be extended using this formula to $z\in \R\setminus\{0,-1,-2,\ldots\}$. 
The Gamma function satisfies in particular the duplication formula (see, \textit{e.g.}, \cite{abramowitz}*{equation 6.1.18})
\begin{align}\label{duplicate}
\sqrt{\pi}\,\Gamma(2z)=2^{2z-1}\Gamma(z)\,\Gamma\Big(z+\frac12\Big)
\qquad\text{for }z>0.
\end{align}
Moreover, it holds (\textit{e.g.}, \cite{abramowitz}*{equation 6.1.17})
\begin{align*}
\Gamma(z)\Gamma(1-z)=\frac{\pi}{\sin(\pi z)}
\qquad\text{for }z\in\R\setminus\mathbb{Z}.
\end{align*}
Furthermore (\textit{e.g.}, \cite{abramowitz}*{equation 6.2.1}),
\begin{align}\label{beta-function}
\int_0^1 t^{z-1}(1-t)^{w-1}\;dt=\int_0^{\infty}\frac{t^{z-1}}{(1+t)^{z+w}}\;dt=\frac{\Gamma(z)\,\Gamma(w)}{\Gamma(z+w)}, \qquad z,w>0.
\end{align}

\subsection{The hypergeometric function} We collect here some facts about the hypergeometric function $\hf$. We suppose in all the following that $a,b,c,z\in\R$ with $c>b>0$ and $z\in[0,1)$, although some formulas might hold in broader generality (we refer to \cite{abramowitz}*{Chapter 15}).

Recall first the integral representation
\begin{align}\label{hyper def}
\hf(a,b;c|z)=\frac{\Gamma(c)}{\Gamma(b)\,\Gamma(c-b)}\int_0^1 t^{b-1}(1-t)^{c-b-1}(1-zt)^{-a}\;dt,
\end{align}
see \cite{abramowitz}*{equation 15.3.1}, and the series expansion
\begin{align}\label{series def}
\hf(a,b;c|z)=\frac{\Gamma(c)}{\Gamma(a)\Gamma(b)}\sum_{k=0}^{\infty} \frac{\Gamma(a+k)\Gamma(b+k)}{\Gamma(c+k)}\frac{z^k}{k!},
\end{align}
see \cite{abramowitz}*{equation 15.1.1}.

In particular one can consider $-a\in\N\cup\{0\}$, in which case one has that $\hf$ reduces to a polynomial of degree $-a$. For example:
\begin{align}
\hf(0,b;c|z)&=1, \label{hyper polyn 0}\\
\hf(-1,b;c|z)&=1-\frac{b}{c}\,z, \label{hyper polyn 1}
\end{align} 
see \cite{abramowitz}*{equation 15.4.1}.

Among the many possible transformations, the following one is important to our purposes:
\begin{align}\label{hyper linear transf}
\hf(a,b;c|z)=(1-z)^{c-a-b}\hf(c-a,c-b;c|z).
\end{align}
Indeed, \eqref{hyper linear transf} alongside \eqref{hyper polyn 0} and \eqref{hyper polyn 1} respectively, bears the following identities (corresponding to the particular cases $c=a$ and $c=a-1$ respectively):
\begin{align}
\hf(a,b;a|z) &= (1-z)^{-b} & \text{if }a>b>0, \label{hyper identity 1} \\
\hf(a,b;a-1|z) &= (1-z)^{-b-1}\Big(1-\frac{a-b-1}{a-1}\,z\Big) \label{hyper identity 2}
& \text{if }a-1>b>0.
\end{align}

Finally, according to \cite{abramowitz}*{formula 15.3.17},
\begin{align}\label{hyper identity 3}
\hf(a,b;2b|z) = 
2^{2a}\big(1+\sqrt{1-z}\big)^{-2a}\hf\bigg(a,a-b+\frac12;b+\frac12\Big|\Big(\frac{1-\sqrt{1-z}}{1+\sqrt{1+z}}\Big)^2\bigg).
\end{align}

\section{A bound on the first eigenvalue}
\label{bound first eigenvalue}

Let $\lambda$ be the first eigenvalue of $(-\Delta)^s$ in $B_1$. A direct bound in terms the first eigenvalue $\lambda_1$ of the classical Dirichlet Laplacian $-\Delta$ on the same ball is given by
\[
\lambda\leq \big(\lambda_1\big)^s,
\]
see \cite{MR3233760}*{Theorem 1.1} or, also, \cites{MR2158176,MR3246044}.
To have a more explicit estimate---which turns out to be a better one for $s$ away from $1$ and $n=1$---recall that the function $u_1\in C^s(\R^n)$, $u_1(x)=\kappa_{n,s}(1-|x|^2)^s_+$, where
\[
\kappa_{n,s}=\frac{\Gamma(n/2)4^{-s}}{\Gamma(1+s)\Gamma(s+\frac{n}{2})},
\]
satisfies $(-\Delta)^s u_1=1$ in $B_1$. In particular, we have
\[
\lambda=\min_{\substack{u\in \cH^s_0(B_1)\\ u\neq 0}} \frac{[u]^2_s}{\|u\|_{L^2(B_1)}^2}\leq \frac{[u_1]^2_s}{\|u_1\|_{L^2(B_1)}^2}.
\]
Here,
\begin{align*}
[u_1]^2_s&=\int_{B_1}u_1(x)\;dx=\kappa_{n,s}\frac{2\pi^{n/2}}{\Gamma(\frac{n}2)}\int_{0}^{1}(1-r^2)^s\,r^{n-1}\;dr\\&=\kappa_{n,s}\frac{\pi^{n/2}}{\Gamma(\frac{n}2)}\int_{0}^{1}(1-t)^s\,t^{\frac{n}{2}-1}\;dt=\kappa_{n,s}\pi^{n/2}\frac{\Gamma(1+s)}{\Gamma(1+s+\frac{n}2)}, \\
\|\tau\|_{L^2(B_1)}^2&=\kappa_{n,s}^2\frac{2\pi^{n/2}}{\Gamma(\frac{n}2)}\int_0^1(1-r^2)^{2s}\,r^{n-1}\;dr=\kappa_{n,s}^2\frac{\pi^{n/2}}{\Gamma(\frac{n}2)}\int_0^1(1-t)^{2s}\,t^{\frac{n}{2}-1}\ dt\\
&=\kappa_{n,s}^2\pi^{n/2}\frac{\Gamma(1+2s)}{\Gamma(1+2s+\frac{n}2)},
\end{align*}
where we used \eqref{beta-function} twice.
Thus,
\begin{equation}\label{new bound on lambda}
\lambda\leq \frac{1}{\kappa_{n,s}}\frac{\Gamma(1+s)\Gamma(1+2s+\frac{n}2)}{\Gamma(1+2s)\Gamma(1+s+\frac{n}2)}=\frac{4^s\Gamma(1+s)^2\,\Gamma(1+2s+\frac{n}{2})}{(s+\frac{n}{2})\,\Gamma(\frac{n}2)\,\Gamma(1+2s)}=\Lambda(n,s).
\end{equation}
In the particular case $n=1$, we have with the properties of the Gamma function (see Appendix \ref{special functions}, in particular \eqref{duplicate})
\begin{align}\label{new bound on lambda n=1}
\lambda\leq\frac{4^s\Gamma(1+s)^2\,\Gamma(\frac32+2s)}{(\frac12+s)\,\Gamma(\frac12)\,\Gamma(1+2s)}
=\frac{s4^s\Gamma(s)^2\,\Gamma(\frac32+2s)}{\sqrt\pi(1+2s)\,\Gamma(2s)}
=\frac{2s\Gamma(s)\,\Gamma(\frac32+2s)}{(1+2s)\,\Gamma(\frac12+s)}
=\frac{\Gamma(1+s)\,\Gamma(\frac32+2s)}{\Gamma(\frac32+s)}.
\end{align}
 
Related results in this direction are contained in Dyda, Kuznetsov, and Kwa\'snicki \cite{MR3656279}.

\section{On the shape of some \texorpdfstring{$s$}{s}-harmonic functions}
\label{minimum at 0}

We discuss here some features of $s$-harmonic functions in $B_1$ associated with particular data in $B_1^c$.
Specifically, we assume
\begin{gather}
g:(1,\infty)\to\R \quad \text{is a non-increasing function,} \label{g-decreasing} \\
\int_{\R^n\setminus B_1}\frac{|g(y)|}{1+{|y|}^{n+2s}}\;dy<\infty. \label{g-integrable}
\end{gather}
We denote by
\begin{align}\label{bt}
b_t=\sup\{y\in(1,\infty):g(y)>t\} \quad \text{for }t\in(-\infty,\overline{g}),
\qquad
\overline g=\lim_{y\downarrow 1}g(y).
\end{align}
Let $h:\R^n\to\R$ be the $s$-harmonic extension of $y\to g(|y|)$ in $B_1$, namely
\begin{align}\label{h}
h(x)=\int_{B_1^c}P_s(x,y)\,g(|y|)\;dy 
=\gamma_{n,s}\int_{B_1^c}\bigg(\frac{1-|x|^2}{|y|^2-1}\bigg)^s\frac{g(|y|)}{{|y-x|}^n}\;dy,
\qquad\text{for }x\in B_1.
\end{align}

\begin{proposition}\label{prop:s-harm}
Assume \eqref{g-decreasing}, \eqref{g-integrable}, and that $g$ is non-negative. 
The function $h$ defined as in \eqref{h} is
radial, radially increasing, and subharmonic in $B_1$.
\end{proposition}
\begin{proof}
Starting from the representation formula \eqref{h}, we write
\begin{align*}
h(x)=\gamma_{n,s}\big(1-|x|^2\big)^s\int_1^\infty\frac{\rho^{n-1}\,g(\rho)}{\big(\rho^2-1\big)^s}\int_{\partial B_1}\frac{d\theta}{\big|\rho\theta-x\big|^n}\;d\rho
\end{align*}
and, in view of Lemma \ref{lem:integral computation}, it holds for $x\in B_1$ and $\rho>1$
\begin{align*}
\int_{\partial B_1}\frac{d\theta}{\big|\rho\theta-x\big|^n} &=
\frac{2\pi^{\frac{n-1}{2}}}{\Gamma(\frac{n-1}{2})}\int_{-1}^1\frac{(1-t^2)^{\frac{n-3}{2}}}{(\rho^2+|x|^2-2\rho|x|t)^{\frac{n}2}}\;dt \\
&=
\frac{2^{n-1}\pi^{\frac{n-1}{2}}}{\Gamma(\frac{n-1}{2})}\int_0^1\frac{\tau^{\frac{n-3}{2}}{(1-\tau)}^{\frac{n-3}{2}}}{\big((\rho+|x|)^2-4\rho|x|\tau\big)^{\frac{n}2}}\;d\tau \\
&=
\frac{2^{n-1}\pi^{\frac{n-1}{2}}}{\Gamma(\frac{n-1}{2})}\frac{\Gamma(\frac{n-1}2)}{\Gamma(n-1)}
\big(\rho+|x|\big)^{-n}\hf\bigg(\frac{n}2,\frac{n-1}2;n-1\Big|\frac{4\rho|x|}{(\rho+|x|)^2}\bigg) \\
&=
\frac{2^{n-1}\pi^{\frac{n-1}{2}}}{}\frac{\Gamma(\frac{n-1}2)}{\Gamma(n-1)}
\rho^{-n}\hf\bigg(\frac{n}2,1;\frac{n}2\Big|\frac{|x|^2}{\rho^2}\bigg)\\
&=
2^{n-1}\pi^{\frac{n-1}{2}}\frac{\Gamma(\frac{n-1}2)}{\Gamma(n-1)}
\frac{\rho^{2-n}}{\rho^2-|x|^2}
\end{align*}
where we have used \eqref{hyper def}, \eqref{hyper identity 3}, and \eqref{hyper identity 1} in this order.
Using the \textit{layer-cake representation} for $g$ 
\begin{align*}
g(\rho)=\int_0^{\overline{g}}\1_{(1,b_t)}(\rho)\;dt,
\qquad \rho>1,
\end{align*}
with $b_t$ and $\overline{g}$ defined as in \eqref{bt}, we write for $x\in B_1$
\begin{align*}
h(x) &=
\gamma_{1,s}\big(1-|x|^2\big)^s\int_1^\infty\frac{2\rho\,g(\rho)}{\big(\rho^2-1\big)^s(\rho^2-|x|^2)}\;d\rho  \\
&=\gamma_{1,s}\big(1-|x|^2\big)^s\int_0^{\overline g}\int_1^{b_t}\frac{2\rho}{\big(\rho^2-1\big)^s(\rho^2-|x|^2)}\;dy\;dt 
\end{align*}
where, for any $b\geq 1$,
\begin{multline*}
\int_1^b\frac{2\rho}{\big(\rho^2-1\big)^s(\rho^2-|x|^2)}\;d\rho
= \int_1^{b^2}\frac{dz}{{(z-1)}^s(z-|x|^2)} 
= \int_0^{b^2-1}\frac{dz}{z^s(z+1-|x|^2)} = \\
= \big(1-|x|^2\big)^{-s}\int_0^{\frac{b^2-1}{1-|x|^2}}\frac{dw}{w^s(w+1)} 
= \big(1-|x|^2\big)^{-s}\int_{\frac{1-|x|^2}{b^2-1}}^\infty\frac{v^{s-1}}{v+1}\;dv.
\end{multline*}
Therefore
\begin{align*}
h(x)=\gamma_{1,s}\int_0^{\overline g}\int_{\frac{1-|x|^2}{b_t^2-1}}^\infty\frac{v^{s-1}}{v+1}\;dv\;dt \qquad \text{for }x\in B_1.
\end{align*}
As a consequence, for any $x\in B_1$,
\begin{align*}
\nabla h(x) &= -\gamma_{1,s}\int_0^{\overline g}\bigg(\frac{1-|x|^2}{b_t^2-1}\bigg)^{s-1}\frac{1}{\frac{1-|x|^2}{b_t^2-1}+1}\,\frac{-2x}{b_t^2-1}\;dt =
2\gamma_{1,s}\,x\big(1-|x|^2\big)^{s-1}\int_0^{\overline g}\frac{\big(b_t^2-1\big)^{1-s}}{b_t^2-|x|^2}\;dt.
\end{align*}
This proves the radial monotonicity. Moreover, for any $x\in B_1$,
\begin{align*}
-\Delta h(x) &=
-2\gamma_{1,s}\;\mathrm{div}\bigg(x\big(1-|x|^2\big)^{s-1}\int_0^{\overline g}\frac{\big(b_t^2-1\big)^{1-s}}{b_t^2-|x|^2}\;dt\bigg) \\
&=
-2n\gamma_{1,s}\,\big(1-|x|^2\big)^{s-1}\int_0^{\overline g}\frac{\big(b_t^2-1\big)^{1-s}}{b_t^2-|x|^2}\;dt 
-4(1-s)\gamma_{1,s}\,|x|^2\big(1-|x|^2\big)^{s-2}\int_0^{\overline g}\frac{\big(b_t^2-1\big)^{1-s}}{b_t^2-|x|^2}\;dt \\
& \qquad -4\gamma_{1,s}\,|x|^2\big(1-|x|^2\big)^{s-1}\int_0^{\overline g}\frac{\big(b_t^2-1\big)^{1-s}}{\big(b_t^2-|x|^2\big)^2}\;dt
\end{align*}
is (strictly) negative for $x\in B_1$. 
\end{proof}
\begin{remark}\label{rmk:increasing}
Analogue calculations can be performed when $g$ is non-negative and non-decreasing instead.
This would lead to a radial, radially decreasing, and super-harmonic $s$-harmonic extension.
\end{remark}
\begin{remark}\label{rmk:sign-changing} The non-negativity assumption in Proposition \ref{prop:s-harm} can be dropped: 
if we split $g=g^+-g^-$ we can directly apply Proposition \ref{prop:s-harm} to $g^+$ and Remark \ref{rmk:increasing} to $g^-$.
\end{remark}

\end{document}